\documentclass[a4paper,12pt,reqno]{amsart}
\usepackage{amsfonts}
\usepackage[colorlinks,linkcolor=blue,citecolor=blue]{hyperref}
\usepackage{latexsym, amssymb, amsmath, amsthm, bbm}
\usepackage[all]{xy}
\usepackage{enumerate}
\usepackage{mathrsfs}
\usepackage{geometry}
\def \k{\Bbbk}
\def \id{\operatorname{id}}

\def \C{\mathcal{C}}
\def \D{\mathcal{D}}

\def \span{\operatorname{span}}

\usepackage{mathtools}

\numberwithin{equation}{section}
\theoremstyle{plain}
 \newtheorem{theorem}{Theorem}[section]

\newtheorem{lemma}[theorem]{Lemma}
\newtheorem{proposition}[theorem]{Proposition}
\newtheorem{corollary}[theorem]{Corollary}
\newtheorem{definition}[theorem]{Definition}
\newtheorem{example}[theorem]{Example}
\newtheorem{remark}[theorem]{Remark}

\begin{document}
\title[Cotensor coquasi-Hopf algebras]{Quasi-quantum groups arising from cotensor coquasi-Hopf algebras}
\thanks{}
\author[J. Yu]{Jing Yu}
\author[X. Zhen]{Xiangjun Zhen}
\address{School of Mathematical Sciences, University of Science and Technology of China, Hefei 230026, China}
\email{yujing46@ustc.edu.cn}
\address{School of Mathematics, Nanjing University, Nanjing 210093, China}
\email{xjzhen@smail.nju.edu.cn}

\begin{abstract}
In this paper, we show that for any coquasi-Hopf algebra $H$ and any $H$-Majid bimodule $M$, the cotensor coalgebra $\operatorname{CoT}_H(M)$ admits a graded coquasi-Hopf algebra structure. We then define the standard filtration of a coquasi-Hopf algebra and show that the associated graded coquasi-Hopf algebra admits an embedding into a cotensor coquasi-Hopf algebra that preserves degrees zero and one. If, in addition, this coquasi-Hopf algebra is coquasitriangular (respectively, coribbon), then the cotensor coquasi-Hopf algebra is also coquasitriangular (respectively, coribbon), and the above embedding is a morphism of coquasitriangular (respectively, coribbon) coquasi-Hopf algebras. We also prove that if a coquasi-antipode $S$ maps the coradical into itself, then $S$ is bijective if and only if its restriction to the coquasi-Hopf subalgebra generated by the coradical is bijective. As a byproduct, we prove that the coquasi-antipode of a coquasi-Hopf algebra with the dual Chevalley property is bijective. Moreover, we establish graded quasi-coquasi-Hopf pairings between tensor quasi-Hopf algebras and cotensor coquasi-Hopf algebras.
\end{abstract}
\date{}
\maketitle

\section{Introduction}
The notions of tensor and cotensor Hopf algebras were introduced by Nichols in his foundational work on bialgebras of type one \cite{Nic78}. Given a Hopf algebra and a Hopf bimodule, these constructions yield a systematic method for constructing graded Hopf algebras and studying their algebraic and coalgebraic structures. In this direction, Rosso \cite{Ros98} realized the positive Borel part of a quantized enveloping algebra as the subalgebra generated by the degree-zero and degree-one components of a suitable cotensor Hopf algebra (namely, the quantum symmetric algebra). Covering quivers \cite{GS98} and Hopf quivers \cite{CR02} were also studied within the framework of tensor and cotensor Hopf algebras. Later, Chen \cite{Che07} established a link between the two structures via a Hopf pairing. Further developments include Fang-Rosso's multi-brace cotensor Hopf algebras \cite{FR17}, realizing quantum groups for symmetrizable Kac-Moody algebras and their highest-weight representations, and Fang-Jian's \cite{FJ15} Hopf algebra structures on cotensor coalgebras from Hopf bimodule algebras, related to quantum quasi-shuffle algebras and Rota-Baxter structures, and encompassing quantum groups and universal Clifford algebras.

Meanwhile, quasi-Hopf algebras were introduced by Drinfeld \cite{Dri90} in connection with the Knizhnik-Zamolodchikov equations of conformal field theory, while their dual counterparts, coquasi-Hopf algebras, were developed by Majid in the context of Tannaka-Krein reconstruction \cite{Maj95}. Following Drinfeld's philosophy of quantum groups \cite{Dri87}, both mutually dual structures can be understood within the framework of quasi-quantum groups. By allowing coassociativity or associativity to be controlled by a (re)associator, these structures extend the Hopf algebra framework and provide algebraic models for tensor categories, conformal field theory, and low-dimensional topology. Recently, Bulacu \cite{Bul20} extended the tensor Hopf algebra construction of Nichols \cite{Nic78} to the quasi-Hopf algebra context. This naturally raises the question of whether an analogous cotensor construction can be developed for coquasi-Hopf algebras.

Moreover, Etingof and Ostrik \cite{EO04} showed that a finite tensor category in which all objects have integer Frobenius-Perron dimensions, called a finite integral tensor category, can be realized from a finite-dimensional quasi-quantum group. Consequently, the classification of such tensor categories reduces to that of the associated quasi-quantum groups. Current classification results for finite-dimensional coquasi-Hopf algebras are mainly restricted to the dual Chevalley case, that is, the case in which the coradical is a coquasi-Hopf subalgebra. Examples include the pointed case \cite{HLYY26}, which extends the classification results for Hopf algebras \cite{AS10}, and the cotriangular case with the dual Chevalley property \cite{EG03,EG21}. In our previous work \cite{Yu26}, we developed a quiver approach to coquasi-Hopf algebras with the dual Chevalley property and studied their corepresentation type.

However, in certain cases, the dual Chevalley property does not hold. For instance, the tensor product structure of the symplectic fermion category shows that the finite-dimensional ribbon quasi-Hopf algebras constructed by Farsad, Gainutdinov, and Runkel \cite{FGR22} fail the Chevalley property; consequently, their duals are coribbon coquasi-Hopf algebras without the dual Chevalley property. From a Hopf-algebraic point of view, Andruskiewitsch and Cuadra \cite{AC13} introduced a new filtration on Hopf algebras, termed the standard filtration, which generalizes the coradical filtration and is designed to investigate more general Hopf algebras. A central idea in their work is to consider the algebra generated by the coradical. This approach was also highlighted in Angiono's report at the 2026 International Congress of Mathematicians \cite{Ang26}, where it was noted that the classification of Hopf algebras without the dual Chevalley property remains an open problem, and that the framework of Andruskiewitsch and Cuadra is expected to play a significant role in future developments. It is therefore natural to ask whether this approach extends to coquasi-Hopf algebras. Moreover, if it does, can cotensor coquasi-Hopf algebras provide an effective framework for studying the resulting graded structures?

In this paper, we first show that for any coquasi-Hopf algebra $H$ and any $H$-Majid bimodule $M$, the cotensor coalgebra $\operatorname{CoT}_H(M)$ admits a graded coquasi-Hopf algebra structure (see Proposition \ref{prop:CotCQH}).
Next, inspired by \cite{AC13}, for a coquasi-Hopf algebra $H$ we define its Majid coradical to be the coquasi-Hopf subalgebra generated by its coradical, denoted by $H_{[0]}$. We then define the standard filtration $\{H_{[n]}\}_{n\geq 0}$ of $H$ by
$$
H_{[n+1]}:=H_{[0]}\wedge H_{[n]}
=\Delta^{-1}(H_{[0]}\otimes H + H\otimes H_{[n]}),\;\; n\geq 0.
$$
We prove that if $S(H_0)\subseteq H_0$, then the associated graded coalgebra
$$
\operatorname{gr}H:=\bigoplus_{n\geq 0}H_{[n]}/H_{[n-1]}
$$
is a graded coquasi-Hopf algebra (see Lemma \ref{lem:gr}). Subsequently, we prove that there exists a graded coquasi-Hopf algebra embedding
\begin{eqnarray}\label{psi}
\psi:\operatorname{gr}H\hookrightarrow \operatorname{CoT}_{H_{[0]}}(H_{[1]}/H_{[0]})
\end{eqnarray}
such that $\psi\mid_{H_{[0]}\oplus H_{[1]}/H_{[0]}}=\operatorname{id}$; see Theorem \ref{thm:embedding}. These results answer the questions raised earlier. As an application, we show that a coquasi-antipode $S$ is bijective if and only if $S|_{H_{[0]}}$ is (see Theorem \ref{thm:H[0]Sbijiective}). Thus, the coquasi-antipode of a (not necessarily finite-dimensional) coquasi-Hopf algebra with the dual Chevalley property is always bijective (see Proposition \ref{prop:Chevalleybijective}). This extends Radford's result \cite[Corollary 3.6]{Rad77} on Hopf algebras with the dual Chevalley property to the coquasi-Hopf algebra setting.

It is well known that the corepresentation category of a coquasitriangular (respectively, coribbon) coquasi-Hopf algebra forms a braided (respectively, ribbon) tensor category. In the finite-dimensional setting, these constructions are dual to those for quasitriangular and ribbon quasi-Hopf algebras. A concrete example of ribbon quasi-Hopf algebras is the twisted quantum double $D^\omega(G)$ of a finite group $G$, constructed by Dijkgraaf, Pasquier and Roche \cite{DPR90}. Note that the representation category of the fixed-point subalgebra $V^G$ of a holomorphic vertex operator algebra $V$ is equivalent to the module category of some twisted Drinfeld double $D^\omega (G)$ \cite{DNR25}. Altsch\"uler and Coste \cite{AC92} constructed invariants of links and three-manifolds from ribbon quasi-Hopf algebras. Creutzig, Gainutdinov, and Runkel \cite{CGR20} gave a ribbon quasi-Hopf algebra $U$, whose underlying algebra is that of the restricted quantum group for $\mathfrak{sl}(2)$ at a $2p$th root of unity, in connection with the triplet VOA $\mathcal{W}(p)$. By taking duals, we obtain examples of coribbon coquasi-Hopf algebras. These connections provide motivation for studying coquasitriangular and coribbon structures on cotensor coquasi-Hopf algebras.
We first study coquasitriangular structures on cotensor coquasi-Hopf algebras (see Propositions \ref{prop:CotHMCQT} and \ref{prop:CQTiff}). We then prove that if $H$ is a coquasitriangular (respectively, coribbon) coquasi-Hopf algebra, then $\operatorname{CoT}_{H_{[0]}}(H_{[1]}/H_{[0]})$ is also coquasitriangular (respectively, coribbon), and the embedding \eqref{psi} is a morphism of coquasitriangular (respectively, coribbon) coquasi-Hopf algebras; see Propositions \ref{prop:CQTembedding} and \ref{prop:coribbonembedding}.

We also give the definition of a quasi-coquasi-Hopf pairing (see Definition \ref{def:QCQpair}), which generalizes that of a Hopf pairing \cite{Kas95}, and establish a connection between quasi-Hopf algebras and coquasi-Hopf algebras. Let $A$ be a quasi-Hopf algebra and $H$ a coquasi-Hopf algebra. Let $M$ be a two-sided two-cosided quasi-Hopf $A$-bimodule and $N$ an $H$-Majid bimodule. We establish a connection between the tensor quasi-Hopf algebra $T_A(M)$ constructed by Bulacu \cite{Bul20} and the cotensor coquasi-Hopf algebra $\operatorname{CoT}_H(N)$ via a quasi-coquasi-Hopf pairing. Specifically, we show that a graded quasi-coquasi-Hopf pairing
$$
\psi:T_A(M)\times \operatorname{CoT}_H(N)\rightarrow \k
$$
exists if and only if there exists a pairing $(\psi_0,\psi_1)$ between $(A,M)$ and $(H,N)$; see Theorem \ref{thm:pair}. This result generalizes that of Chen \cite{Che07}.

The organization of this paper is as follows: In Section \ref{section2}, we recall the definitions of coquasi-Hopf algebras, Majid bimodules, Yetter-Drinfeld modules, and cotensor coalgebras, along with some of their basic properties. We prove that, for a coquasi-Hopf algebra $H$ and a Majid bimodule $M$, the cotensor coalgebra $\operatorname{CoT}_H(M)$ can be endowed with a graded coquasi-Hopf algebra structure. In Section \ref{section3}, we first give the definitions of the Majid coradical and the standard filtration. We then construct a graded coquasi-Hopf algebra associated with this filtration and show that it admits an embedding into a cotensor coquasi-Hopf algebra that preserves degrees zero and one. As an application, we investigate the bijectivity of the coquasi-antipode. In Section \ref{section4}, we show that if $H$ is a coquasitriangular (respectively, coribbon) coquasi-Hopf algebra, then $\operatorname{CoT}_{H_{[0]}}(H_{[1]}/H_{[0]})$ is also coquasitriangular (respectively, coribbon), and the embedding is a morphism of coquasitriangular (respectively, coribbon) coquasi-Hopf algebras. Section \ref{Section5} is devoted to establishing quasi-coquasi-Hopf pairings between tensor quasi-Hopf algebras and cotensor coquasi-Hopf algebras.

\section{On cotensor coquasi-Hopf algebras}\label{section2}
We work over a field $\k$, and all vector spaces are over $\k$. The tensor product over $\k$ will be denoted simply by $\otimes$. We refer the reader to \cite{EGNO15} for tensor categories and to \cite{BCPV19} for coquasi-Hopf algebras.
\subsection{Coquasi-Hopf algebras}
Coquasi-Hopf algebras (see \cite[Section 2.4]{Maj95}), also referred to as Majid algebras or dual quasi-Hopf algebras, constitute the dual notion of Drinfeld's quasi-Hopf algebras \cite{Dri90}.

\begin{definition}
A coquasi-bialgebra is a coalgebra $(H,\Delta,\varepsilon)$ over $\k$ endowed with a compatible quasi-algebra structure. Concretely, this amounts to the data of two coalgebra homomorphisms
$$
m: H \otimes H \rightarrow H, \;\; a \otimes b \mapsto ab,
\;\;\text{and}\;\;
\mu: \k \rightarrow H, \;\; \lambda \mapsto \lambda 1_H,
$$
along with a convolution-invertible reassociator $\Phi: H^{\otimes 3} \rightarrow \k$, subject to the following conditions for arbitrary $a,b,c,d \in H$:
\begin{equation}
a_1(b_1c_1)\Phi(a_2,b_2,c_2)=\Phi(a_1,b_1,c_1)(a_2b_2)c_2, \label{eq:associativity}
\end{equation}
\begin{equation}
1_H a = a = a 1_H, \label{eq:1h=h1=h}
\end{equation}
\begin{equation}
\Phi(a_1,b_1,c_1d_1)\Phi(a_2b_2,c_2,d_2)
= \Phi(b_1,c_1,d_1)\Phi(a_1,b_2c_2,d_2)\Phi(a_2,b_3,c_3), \label{eq:penta}
\end{equation}
\begin{equation}
\Phi(a,1_H,b)=\varepsilon(a)\varepsilon(b). \label{eq:Phi(1)=e}
\end{equation}
The coquasi-bialgebra $H$ is termed a coquasi-Hopf algebra whenever it possesses a coquasi-antipode: this consists of a coalgebra antimorphism $S : H \to H$ together with two maps $\alpha, \beta : H \rightarrow \Bbbk$ satisfying, for every $a \in H$,
\begin{equation}
S(a_1)\alpha(a_2)a_3 = \alpha(a)1_H, \;\;
a_1\beta(a_2)S(a_3) = \beta(a)1_H, \label{eq:alpha,beta}
\end{equation}
and additionally
\begin{equation}
\Phi(a_1,S(a_3),a_5)\beta(a_2)\alpha(a_4)
= \Phi^{-1}(S(a_1),a_3,S(a_5)) \alpha(a_2)\beta(a_4)
= \varepsilon(a). \label{eq:Phialphabeta=e}
\end{equation}
Throughout, we employ the shorthand $\Phi(a,b,c):=\Phi(a\otimes b\otimes c)$ and adopt Sweedler's notation $\Delta(a)=a_1 \otimes a_2$ for the coproduct; the $n$-fold iterated coproduct of $a$ is written as $a_1 \otimes a_2 \otimes \cdots \otimes a_{n+1}$. Finally, the coquasi-antipode $(S,\alpha,\beta)$ is called bijective when the underlying map $S$ is bijective.
\end{definition}
Note that identities (\ref{eq:penta}) and (\ref{eq:Phi(1)=e}) force the normalizations
\begin{equation}\label{eq:Phi(1,a,b)=e(a)e(b)}
\Phi(1_H,a,b)=\Phi(a,b,1_H)=\varepsilon(a)\varepsilon(b),\;\; \forall\; a,b\in H.
\end{equation}
From (\ref{eq:alpha,beta}) and (\ref{eq:Phialphabeta=e}) one further deduces that $S(1_H)=1_H$ and $\alpha(1_H)\beta(1_H)=1$.
It should be emphasized that the coquasi-antipode is determined only up to a convolution-invertible element $U\in H^*$.
Moreover, it is a standard fact that the antipode of a Hopf algebra is an algebra antimorphism. Analogously, for a coquasi-Hopf algebra $H$, its coquasi-antipode behaves as an algebra antimorphism up to twist. Specifically, for any $ h, g \in H $, set
\begin{equation}\label{eq:gamma}
\gamma(h,g) := \Phi(S(g_2), S(h_2), h_4) \alpha(h_3) \Phi^{-1}(S(g_1)S(h_1), h_5, g_4) \alpha(g_3),
\end{equation}
and then define
\begin{equation}\label{eq:f}
f(h,g) := \Phi^{-1}(S(g_1)S(h_1), h_3g_3, S(h_5g_5)) \beta(h_4g_4) \gamma(h_2, g_2).
\end{equation}
Then $f$ is convolution-invertible, and one has
\begin{equation}\label{eq:S(hg)}
f(h_1, g_1)S(h_2g_2)f^{-1}(h_3, g_3) = S(g)S(h).
\end{equation}
For details, see \cite[Proposition 3.49]{BCPV19}.

\begin{example}\rm\label{ex:group}
Given a group $G$, its group algebra $\k G$ carries a natural Hopf algebra structure, with comultiplication, counit, and antipode given by
$$
\Delta(g)=g\otimes g,\;\; \varepsilon(g)=1,\;\; S(g)=g^{-1}
$$
for all $g\in G$. Let $\omega$ be a normalized $3$-cocycle on $G$, i.e.,
$$
\omega(ef,g,h)\omega(e,f,gh)=\omega(e,f,g)\omega(e,fg,h)\omega(f,g,h),\;\; \omega(f,1,g)=1
$$
for all $e,f,g,h\in G$. Extending linearly, $\omega:(\k G)^{\otimes 3}\to \k$ becomes convolution-invertible. Define linear maps $\alpha,\beta:\k G\to \k$ by
$$
\alpha(g):=\varepsilon(g),\;\; \beta(g):=\frac{1}{\omega(g,g^{-1},g)}
$$
for all $g\in G$. Then $(\k G,\omega, S, \alpha,\beta)$ is a cosemisimple coquasi-Hopf algebra, denoted simply by $(\k G,\omega)$ (see, for example, \cite[Example 2.2]{HLYY20}). Its category of comodules is precisely the Gr-category $\operatorname{Vec}^{\omega}_G$ (see \cite[Example 2.3.8]{EGNO15}).
\end{example}

\subsection{Majid bimodules and Yetter-Drinfeld modules}
In this subsection, let $H$ be a coquasi-bialgebra over $\k$ with reassociator $\Phi$. We first give the definitions of Majid bimodules and Yetter-Drinfeld modules.
\begin{definition}\emph{(}\cite[Definition 3.2]{Hua09}\emph{)}
A linear space $M$ is called an $H$-Majid bimodule if $M$ is an $H$-bicomodule with structure maps $(\rho_M^{L},\rho_M^{R})$, and there exist two $H$-bicomodule morphisms
$$
p_L: H \otimes M \rightarrow M, \;\; h \otimes m \mapsto h\cdot m,
\;\;
p_R: M \otimes H \rightarrow M, \;\; m \otimes h \mapsto m\cdot h
$$
such that for all $g, h \in H$ and $m \in M$, the following identities hold:
\begin{equation*}
1_H\cdot m=m=m\cdot 1_H,
\end{equation*}
\begin{equation*}
g_1\cdot(h_1\cdot m_0)\Phi(g_2,h_2,m_1)=\Phi(g_1,h_1,m_{-1})(g_2h_2)\cdot m_0,
\end{equation*}
\begin{equation*}
m_0\cdot (g_1h_1)\Phi(m_1,g_2,h_2)=\Phi(m_{-1},g_1,h_1)(m_0\cdot g_2)\cdot h_2,
\end{equation*}
\begin{equation*}
g_1\cdot (m_0\cdot h_1)\Phi(g_2,m_1,h_2)=\Phi(g_1,m_{-1},h_1)(g_2\cdot m_0)\cdot h_2.
\end{equation*}
\end{definition}

\begin{definition}\emph{(}\cite[Definition 3.1]{Bal09}\emph{)} \label{def:YDmod}
A left-left Yetter-Drinfeld module over $H$ is a triple $( V,\rho^L_V ,\vartriangleright )$, where
\begin{itemize}
\item $(V,\rho^L_V)$ is a left $H$-comodule;
\item $\vartriangleright : H\otimes V \rightarrow V$ is a $\k$-linear map satisfying, for all $h,l \in H$ and $v \in V$, the following conditions:
\begin{equation*}
(hl)\vartriangleright v
= \left[ \begin{array}{l}
\Phi^{-1}(h_1,l_1,v_{-1})\Phi(h_2,(l_2\vartriangleright v_0)_{-1},l_3)\\
   \Phi^{-1}((h_3\vartriangleright(l_2\vartriangleright v_0)_0)_{-1},h_4,l_4)\bigl(h_3\vartriangleright(l_2\vartriangleright v_0)_0\bigr)_0
   \end{array} \right],
\end{equation*}
\begin{equation*}
1_H\vartriangleright v = v,
\end{equation*}
\begin{equation*}
(h_1\vartriangleright v)_{-1}h_2 \otimes (h_1\vartriangleright v)_0
= h_1 v_{-1} \otimes (h_2\vartriangleright v_0).
\end{equation*}
\end{itemize}
\end{definition}
Let $H$ be a coquasi-bialgebra over $\k$. By \cite[Theorem 1.12]{AMS07}, the category $({}_H^{H}\mathscr{M}^{H}_H, \otimes_H, H)$ of $H$-Majid bimodules is a monoidal category.
Ardizzoni and Pavarin \cite{AP13} established a connection between the category ${}_H^{H}\mathscr{M}^{H}_H$ of $H$-Majid bimodules and the category ${}_H^H\mathcal{YD}$ of Yetter-Drinfeld modules over $H$. They proved that if $H$ admits a preantipode (in the sense of \cite[Definition 3.6]{AP12}), then the category $({}_H^{H}\mathscr{M}^{H}_H, \otimes_H, H)$ of $H$-Majid bimodules is monoidally equivalent to the category $({}^H_H\mathcal{YD}, \otimes,\k)$ of Yetter-Drinfeld modules over $H$ (see \cite[Proposition 3.8 and Lemma 4.4]{AP13}).
Denoting by $\Box$ the cotensor product introduced in \cite{Tak77}, the category $({}_H^{H}\mathscr{M}^{H}_H, \otimes_H, H)$ is monoidally equivalent to the category $({}_H^{H}\mathscr{M}^{H}_H, \Box, H)$ by \cite[Corollary 4.12]{AP13}, and the equivalence is given by the identity functor.
Since every coquasi-Hopf algebra admits a preantipode (\cite[Theorem 3.10]{AP12}), the following lemma holds.
\begin{lemma}\label{lem:Majidmod=YD}\emph{(}\cite[Proposition 2.23]{Yu26}\emph{)}
Let $H$ be a coquasi-Hopf algebra over $\k$. Then the functors
$$
G:\; ({}_H^{H}\mathscr{M}^{H}_H, \Box, H) \rightarrow  ({}^H_H\mathcal{YD}, \otimes,\k),\;\;
M\mapsto M^{coH}
$$
and
$$
F:({}^H_H\mathcal{YD}, \otimes,\k)\rightarrow({}_H^{H}\mathscr{M}^{H}_H, \Box, H),\;\;
N\mapsto N\otimes H
$$
form a pair of mutually inverse monoidal equivalences, where $M^{co H}= \{m\in M\mid \rho_M^R(m)=m\otimes 1_H\}$. More explicitly, $M^{coH}\in {}^H_H\mathcal{YD}$ and $N\otimes H\in {}_H^{H}\mathscr{M}^{H}_H$ are equipped with the following structures: for all $m\in M^{coH}$, $n\in N$, and $h,l\in H$,
\begin{equation*}
\rho^L _{M^{coH}}(m) := \rho^L_M(m),
\end{equation*}
\begin{equation*}
h\vartriangleright m:=\beta(h_3)\alpha(h_6)
\Phi(h_1m_{-1},S(h_5),h_7)
(h_2m_0)S(h_4),
\end{equation*}
\begin{equation*}
l\cdot (n\otimes h) := \Phi(l_1, n_{-1}, h_1)\Phi^{-1}((l_2\vartriangleright n_0)_{-1}, l_3, h_2)(l_2\vartriangleright n_0)_0 \otimes l_4 h_3,
\end{equation*}
\begin{equation*}
(n\otimes h)\cdot l := \Phi^{-1}(n_{-1}, h_1, l_1) n_0 \otimes h_2 l_2,
\end{equation*}
\begin{equation*}
\rho_{N\otimes H}^{L}(n\otimes h) := n_{-1} h_1 \otimes (n_0 \otimes h_2),
\end{equation*}
\begin{equation*}
\rho_{N\otimes H}^{R}(n\otimes h) := (n\otimes h_1) \otimes h_2.
\end{equation*}
\end{lemma}

\subsection{Cotensor coalgebras}
Given a coalgebra $H$ and an $H$-bicomodule $M$ with structure maps $\rho^L_M$ and $\rho^R_M$, Nichols defines the \textit{cotensor coalgebra} (see \cite[p. 1526]{Nic78})
\begin{equation}\label{eq:Cotgrading}
\operatorname{CoT}_H(M):=H\oplus M\oplus M^{\Box 2}\oplus \cdots \oplus M^{\Box n}\oplus \cdots,
\end{equation}
where $M^{\Box 0}:=H$, $M^{\Box 1}:=M$, and for each $n\geq 2$, $M^{\Box n}$ is the kernel of the $\k$-linear map
$$
M^{\otimes n}\xrightarrow{f}\bigoplus\limits_{i=1}^{n-1}M^{\otimes i}\otimes H \otimes M^{\otimes n-i},
$$
with $f=(f_1, f_2, \cdots , f_{n-1})$ and $f_i=\id^{\otimes i-1}\otimes \rho^R_M\otimes \id^{\otimes n-i}-\id^{\otimes i}\otimes \rho_M^L\otimes \id^{\otimes n-i-1}$ for $i=1, 2, \cdots, n-1.$
The coalgebra structure on $\operatorname{CoT}_H(M)$ is given by
\begin{equation*}
\varepsilon \mid_H=\varepsilon_H,\;\;\varepsilon \mid_{M^{\Box i}}=0\;\; (i\geq 1),\;\;
\Delta\mid_H=\Delta_H,\;\;\Delta\mid_M=\rho^L_M+\rho^R_M,
\end{equation*}
and for $\sum m^1\otimes m^2\otimes \cdots \otimes m^n\in M^{\Box n}$, $n\geq 2,$
\begin{equation*}
\begin{aligned}
\Delta\left(\sum m^1\otimes m^2\otimes \cdots \otimes m^n\right)
:=&  \sum \rho_M^L(m^1)\otimes  m^2\otimes \cdots \otimes m^n  \\& + \sum m^1\otimes (m^2\otimes \cdots \otimes m^n)+\cdots \\
&  + \sum(m^1\otimes m^2\otimes \cdots \otimes m^{n-1})\otimes m^n \\&+  \sum m^1\otimes m^2\otimes \cdots \otimes m^{n-1}\otimes \rho_M^R(m^n) \\
\in&  \bigoplus\limits_{i=0}^{n}M^{\Box i} \otimes M^{\Box (n-i)}.
\end{aligned}
\end{equation*}

The following universal property holds.

\begin{lemma}\emph{(}\cite[Lemma 1.2]{VZ04}\emph{)} \label{lemma:universal}
Let $H$ be a coalgebra, $M$ an $H$-bicomodule, and $\psi:X\rightarrow \operatorname{CoT}_H(M)$ a coalgebra map. For each $n\geq 0$, denote by $\psi_n:=p_n\psi:X\rightarrow M^{\Box n}$ the projection onto the $n$-th component. Then $\psi_0:X\rightarrow H$ is a coalgebra map, and $\psi_1:X\rightarrow M$ is an $H$-bicomodule map, where the $H$-bicomodule structure on $X$ is induced via $\psi_0$. For $n\geq2$, $\psi_n$ coincides with the $H$-bicomodule map given by the composition
$$
X\xrightarrow{\Delta^{n-1}} X^{\otimes n}\xrightarrow{\psi_1^{\otimes n}} M^{\otimes n}.
$$
Consequently, $\psi$ is uniquely determined by $\psi_0$ and $\psi_1$.
Conversely, suppose $\psi_0:X\rightarrow H$ is a coalgebra map and $\psi_1:X\rightarrow M$ is an $H$-bicomodule map. Define $\psi_n:X\rightarrow M^{\otimes n}$ by the above composition. Then each $\psi_n$ is an $H$-bicomodule map with $\operatorname{Im}(\psi_n)\subseteq M^{\Box n}$. If for every $x\in X$ only finitely many $\psi_i(x)$ are nonzero, then
$\psi:=\sum_{i\geq0}\psi_i:X\rightarrow \operatorname{CoT}_H(M)$ is a coalgebra map.
\end{lemma}

\subsection{Coquasi-Hopf algebra structures on cotensor coalgebras}\label{subsection3.1}
Nichols \cite{Nic78} proved that, for a Hopf algebra $H$ and a Hopf bimodule $M$, the cotensor coalgebra $\operatorname{CoT}_H(M)$ can be endowed with a Hopf algebra structure, using the universal property stated in Lemma \ref{lemma:universal}. In this subsection, we extend these results to coquasi-Hopf algebras.

First, starting from a coquasi-bialgebra $H$ and an $H$-Majid bimodule $M$, we can prove that $\operatorname{CoT}_H(M)$ admits a graded coquasi-bialgebra structure.
\begin{lemma}\label{lem:cotensor=coquasibialgebra}
Let $(H, m, \mu, \Delta, \varepsilon, \Phi)$ be a coquasi-bialgebra over $\k$, and let $M$ be an $H$-Majid bimodule. Then $\operatorname{CoT}_H(M)$ admits a coquasi-bialgebra structure with respect to the grading (\ref{eq:Cotgrading}).
\end{lemma}
\begin{proof}
Let $m^\prime_0$ denote the composition
$$
\operatorname{CoT}_H(M)\otimes \operatorname{CoT}_H(M)
\xrightarrow{\pi_0\otimes \pi_0}
H\otimes H
\xrightarrow{m}
H,
$$
and let $m^\prime_1$ be given by the composition
$$
\operatorname{CoT}_H(M)\otimes \operatorname{CoT}_H(M)
\xrightarrow{\pi_0\otimes \pi_1\oplus \pi_1\otimes \pi_0}
(H\otimes M)\oplus (M\otimes H)
\xrightarrow{p_L\oplus p_R}
M,
$$
where $\pi_0$ and $\pi_1$ denote the canonical projections. For each $n\geq 2$, we define $m^\prime_n$ as the composition
$$
\operatorname{CoT}_H(M)\otimes \operatorname{CoT}_H(M)
\xrightarrow{\Delta^{n-1}}
\bigl(\operatorname{CoT}_H(M)\otimes \operatorname{CoT}_H(M)\bigr)^{\otimes n}
\xrightarrow{(m^\prime_1)^{\otimes n}}
M^{\otimes n}.
$$
We then set $m^\prime:=\sum_{n\geq 0} m_n^\prime$. In addition, we define $\Phi^\prime: \bigl(\operatorname{CoT}_H(M)\bigr)^{\otimes 3}\rightarrow \k$ by
$$
\Phi^\prime(a\otimes b\otimes c):=
\begin{cases}
\Phi(a,b,c), & \text{if } a,b,c \in H,\\
0, & \text{otherwise}.
\end{cases}
$$
By Lemma \ref{lemma:universal}, the coalgebra $\operatorname{CoT}_H(M)$, together with the unit $1_H$, the multiplication $m^\prime$ just defined, and the reassociator $\Phi^\prime$, admits the structure of a coquasi-bialgebra.
\end{proof}

Let $H$ be a graded coquasi-bialgebra. Its reassociator $\Phi$ is said to be \textit{concentrated in degree $0$} if $\Phi(x,y,z)=0$ for all homogeneous elements $x,y,z$ unless $\deg(x)=\deg(y)=\deg(z)=0$.

\begin{lemma}\label{lem:gradedMajid}
Let $H = \bigoplus_{n\geq 0} H(n)$ be a graded coquasi-bialgebra whose reassociator is concentrated in degree $0$. Then $H(1)$ admits an $H(0)$-Majid bimodule structure.
\end{lemma}
\begin{proof}
Let $\pi: H \rightarrow H(1)$ denote the quotient map. For any $x\in H(1)$, define
\begin{eqnarray*}\label{comodulestructure}
\rho^L_{H(1)}(x)=(\id\otimes\pi)\Delta(x),\;\;\rho^R_{H(1)}(x)=(\pi\otimes\id)\Delta(x).
\end{eqnarray*}
A direct verification shows that $(H(1), \rho^L, \rho^R)$ is an $H(0)$-bicomodule. Moreover, we equip $H(1)$ with an $H(0)$-bimodule structure by setting
$$ h\cdot x:=hx,\; \;  x\cdot h:=xh, \;\; \text{for all }h\in H(0),x\in H(1).$$
Together with the bicomodule and bimodule structures defined above, $H(1)$ is an $H(0)$-Majid bimodule.
\end{proof}

By Lemmas \ref{lem:cotensor=coquasibialgebra} and \ref{lem:gradedMajid}, we immediately have the following corollary.
\begin{corollary}\label{coro:coquaibialgiff}
Let $H$ be a coalgebra over $\k$, and let $M$ be an $H$-bicomodule. Then $\operatorname{CoT}_H(M)$ admits a graded coquasi-bialgebra structure with respect to the grading (\ref{eq:Cotgrading}) whose reassociator is concentrated in degree $0$ if and only if $H$ is a coquasi-bialgebra and $M$ is an $H$-Majid bimodule.
\end{corollary}

Next, we prove that if, in addition, $H$ is a coquasi-Hopf algebra, then $\operatorname{CoT}_H(M)$ also admits a coquasi-Hopf algebra structure.
\begin{proposition}\label{prop:CotCQH}
Let $H$ be a coquasi-Hopf algebra over $\k$, and let $M$ be an $H$-Majid bimodule. Then $\operatorname{CoT}_H(M)$ admits a graded coquasi-Hopf algebra structure with respect to the grading (\ref{eq:Cotgrading}).
\end{proposition}
\begin{proof}
Suppose $H$ is a coquasi-Hopf algebra with a coquasi-antipode $(S,\alpha,\beta)$. Using Lemma \ref{lem:cotensor=coquasibialgebra}, we obtain a graded coquasi-bialgebra structure on $\operatorname{CoT}_H(M)$. To this end, we extend $\alpha$ and $\beta$ to functions on $\operatorname{CoT}_H(M)$ by setting $\alpha(a)=\beta(a)=0$ for every homogeneous element $a\in \operatorname{CoT}_H(M)$ with $a\notin H$, and denote the resulting extensions by $\alpha^\prime$ and $\beta^\prime$. Next, applying the universal property of the cotensor coalgebra from Lemma \ref{lemma:universal}, we construct
$S^\prime: \operatorname{CoT}_H(M)\rightarrow \operatorname{CoT}_H(M)$
and show that $(S^\prime, \alpha^\prime, \beta^\prime)$ is a coquasi-antipode of $\operatorname{CoT}_H(M)$. Let $X=(\operatorname{CoT}_H(M))^{cop}$, and define $S^\prime_0$ as the composition
$$
X\xrightarrow{\pi_0} H\xrightarrow{S}H.
$$
Clearly, $S^\prime_0$ is a coalgebra map. Next, we construct a map $T:M\rightarrow M$ such that the composition
$$
X\xrightarrow{\pi_1}M\xrightarrow{T}M
$$
is an $H$-bicomodule map; we denote this composition by $S^\prime_1$.
According to \cite[Proposition 2.16]{Yu26}, we have
$$
H\otimes {}^{co H}M\cong M,\;\; h\otimes m\mapsto h\cdot m,
$$
where ${}^{co H}M= \{m\in M\mid \rho_M^L(m)=1_H\otimes m\}$. This implies that we only need to define $T$ on elements of the form $h\cdot m$, with $h\in H$ and $m\in {}^{co H}M$. Now we define
$$
T(h\cdot m):=-\beta(m_1) (m_0 \cdot S(m_2))\cdot S(h_1)f(h_2, m_3),
$$
where the map $f(h,g)$ is defined in (\ref{eq:f}).
A straightforward computation yields
\begin{equation*}
\begin{aligned}
\rho_M^R\bigl(T(h\cdot m)\bigr)
&=
-\bigl(m_0\cdot S(m_4)\bigr)\cdot S(h_2)
\otimes
m_1\beta(m_2)S(m_3)S(h_1)
f(h_3,m_5)
\\
&=
-\beta(m_1)
\bigl(m_0\cdot S(m_2)\bigr)\cdot S(h_2)
f(h_3,m_3)
\otimes S(h_1)
\\
&=
T(h_2\cdot m)\otimes S(h_1)
\\
&=
T\bigl((h\cdot m)_0\bigr)
\otimes
S\bigl((h\cdot m)_{-1}\bigr).
\end{aligned}
\end{equation*}
By \cite[Proposition~3.49]{BCPV19}, we have
\begin{equation*}
\begin{aligned}
\rho_M^L\bigl(T(h\cdot m)\bigr)
&=
-\beta(m_1)S(m_3)S(h_2)
\otimes
\bigl(m_0\cdot S(m_2)\bigr)\cdot S(h_1)
f(h_3,m_4)
\\
&=
-\beta(m_1)S(h_3m_4)
\otimes
\bigl(m_0\cdot S(m_2)\bigr)\cdot S(h_1)
f(h_2,m_3)
\\
&=
S(h_2m_1)\otimes T(h_1\cdot m_0)
\\
&=
S\bigl((h\cdot m)_1\bigr)
\otimes
T\bigl((h\cdot m)_0\bigr).
\end{aligned}
\end{equation*}
Hence $S^\prime_1$ is an $H$-bicomodule map.
For $n\geq 2, $ define $S^\prime_n$ as the composition
$$
(\operatorname{CoT}_H(M))^{cop}\xrightarrow{\Delta^{n-1}} ((\operatorname{CoT}_H(M))^{cop})^{\otimes n} \xrightarrow{(S^\prime_1)^{\otimes n}} M^{\otimes n}.
$$
Define $S^\prime = \sum_{i=0} S^\prime_i$. Since the restriction of $S^\prime_n$ to $M^{\Box m}$ is zero whenever $m \neq n$, Lemma \ref{lemma:universal} implies that
$S^\prime: \operatorname{CoT}_H(M)\rightarrow \operatorname{CoT}_H(M)$
is a coalgebra antimorphism. Since (\ref{eq:Phialphabeta=e}) holds by construction, it remains to verify (\ref{eq:alpha,beta}). Similar to the proof of \cite[Theorem 4.12]{Yu26}, we first prove that every element of $M$ satisfies (\ref{eq:alpha,beta}), and then use the definition to directly show that all elements of $M^{\Box n}$ also satisfy (\ref{eq:alpha,beta}) for any $n$. Since $\alpha^\prime(M)=\beta^\prime(M)=0$, it suffices to show that for any $h\in H$ and $m\in {}^{co H}M $, we have
\begin{eqnarray*}
(\operatorname{id} * \beta^\prime * S^\prime)(h \cdot m)
=0=
(S^\prime * \alpha^\prime * \operatorname{id})(h \cdot m).
\end{eqnarray*}
First we have
\begin{eqnarray*}
&&(\operatorname{id} * \beta^\prime * S^\prime)(h \cdot m)
\\
&=&
\beta(h_2)h_1\cdot T(h_3\cdot m)
+\beta(h_2m_1)(h_1\cdot m_0)\cdot S(h_3m_2)
\\
&=&
-\beta(m_1)\beta(h_2)f(h_4,m_3)
h_1\cdot\bigl((m_0\cdot S(m_2))\cdot S(h_3)\bigr)
+\beta(h_2m_1)(h_1\cdot m_0)\cdot S(h_3m_2)
\\
&=&
-\beta(m_2)\beta(h_2)f(h_5,m_5)
\Phi(m_1,S(m_3),S(h_3))
h_1\cdot\bigl(m_0\cdot(S(m_4)S(h_4))\bigr)
\\
&&
+\beta(h_2m_1)(h_1\cdot m_0)\cdot S(h_3m_2)
\\
&=&
-\beta(m_2)\beta(h_2)f(h_4,m_4)
\Phi(m_1,S(m_3),S(h_3))
h_1\cdot\bigl(m_0\cdot S(h_5m_5)\bigr)
\\
&&
+\beta(h_2m_1)(h_1\cdot m_0)\cdot S(h_3m_2)
\\
&=&
-\beta(m_3)\beta(h_3)f(h_5,m_5)
\Phi(m_2,S(m_4),S(h_4))
\Phi^{-1}(h_2,m_1,S(h_6m_6))(h_1\cdot m_0)
\\&& \cdot S(h_7m_7)
+\beta(h_2m_1)(h_1\cdot m_0)\cdot S(h_3m_2)
\\
&=&
-\beta(m_3)\beta(h_3)f(h_6,m_6)
\Phi(m_2,S(m_4),S(h_4))
\Phi^{-1}(h_2,m_1,S(m_5)S(h_5))
(h_1\cdot m_0)
\\&&\cdot S(h_7m_7)+\beta(h_2m_1)(h_1\cdot m_0)\cdot S(h_3m_2)
\\
&=&
-\beta(m_4)\beta(h_5)f(h_8,m_8)
\Phi(h_2m_1,S(m_7),S(h_7))
\Phi^{-1}(h_3,m_2,S(m_6))
\\
&&
\Phi^{-1}(h_4,m_3S(m_5),S(h_6))
(h_1\cdot m_0)\cdot S(h_9m_9)
+\beta(h_2m_1)(h_1\cdot m_0)\cdot S(h_3m_2)
\\
&=&
-\beta(m_3)\beta(h_4)f(h_6,m_6)
\Phi(h_2m_1,S(m_5),S(h_5))
\Phi^{-1}(h_3,m_2,S(m_4))(h_1\cdot m_0)\\
&&
\cdot S(h_7m_7)+\beta(h_2m_1)(h_1\cdot m_0)\cdot S(h_3m_2)
\\
&=&
-\delta(h_2,m_1)f(h_3,m_2)
(h_1\cdot m_0)\cdot S(h_4m_3)
+\beta(h_2m_1)(h_1\cdot m_0)\cdot S(h_3m_2)
\\
&=&
-\beta(h_2m_1)(h_1\cdot m_0)\cdot S(h_3m_2)
+\beta(h_2m_1)(h_1\cdot m_0)\cdot S(h_3m_2)
\\
&=&0,
\end{eqnarray*}
where $\delta(h,g):=\Phi(h_1g_1, S(g_5), S(h_4))\beta(h_3)\Phi^{-1}(h_2,g_2,S(g_4))\beta(g_3).$
Similarly, we have
\begin{eqnarray*}
&&(S^\prime * \alpha^\prime * \operatorname{id})(h\cdot m)
\\
&=&
\alpha(h_2)S(h_1)\cdot(h_3\cdot m)
+\alpha(h_2m_1)T(h_1\cdot m_0)\cdot h_3m_2
\\
&=&
\alpha(h_2)S(h_1)\cdot(h_3\cdot m)
-\beta(m_1)f(h_2,m_3)\alpha(h_3m_4)
\bigl((m_0\cdot S(m_2))\cdot S(h_1)\bigr)\cdot h_4m_5
\\
&=&
\alpha(h_3)\bigl(S(h_2)h_4\cdot m_0\bigr)
\Phi^{-1}(S(h_1),h_5,m_1)
-\beta(m_2)f(h_3,m_5)\alpha(h_4m_6)\\
&&
\Phi(m_1,S(m_3),S(h_1))
\bigl(m_0\cdot(S(m_4)S(h_2))\bigr)\cdot h_5m_7
\\
&=&
\alpha(h_2)\Phi^{-1}(S(h_1),h_3,m_1)m_0
-\beta(m_2)f(h_2,m_4)\alpha(h_4m_6)\Phi(m_1,S(m_3),S(h_1))\\
&&
\bigl(m_0\cdot S(h_3m_5)\bigr)\cdot h_5m_7
\\
&=&
\alpha(h_2)\Phi^{-1}(S(h_1),h_3,m_1)m_0
-\beta(m_3)f(h_2,m_5)\alpha(h_5m_8)\Phi(m_2,S(m_4),S(h_1))\\
&&
\Phi(m_1,S(h_3m_6),h_7m_{10})
\cdot
m_0\cdot\bigl(S(h_4m_7)(h_6m_9)\bigr)
\\
&=&
\alpha(h_2)\Phi^{-1}(S(h_1),h_3,m_1)m_0
-\beta(m_3)f(h_2,m_5)\alpha(h_4m_7)\Phi(m_2,S(m_4),S(h_1))\\
&&
\Phi(m_1,S(h_3m_6),h_5m_8)m_0
\\
&=&
\alpha(h_2)\Phi^{-1}(S(h_1),h_3,m_1)m_0
-\beta(m_3)f(h_3,m_6)\alpha(h_4m_7)\Phi(m_2,S(m_4),S(h_1))\\
&&
\Phi(m_1,S(m_5)S(h_2),h_5m_8)m_0
\\
&=&
\alpha(h_2)\Phi^{-1}(S(h_1),h_3,m_1)m_0
-\beta(m_3)\gamma(h_3,m_6)\Phi(m_2,S(m_4),S(h_1))\\
&&
\Phi(m_1,S(m_5)S(h_2),h_4m_7)m_0
\end{eqnarray*}
From the proof of \cite[Theorem 4.12]{Yu26}, we obtain
\begin{eqnarray*}
\Phi(g_1, S(g_{5})S(h_2),h_4g_7)\Phi(g_2,S(g_4),S(h_1))\beta(g_3)\gamma(h_3, g_6)=\alpha(h_2)\Phi^{-1}(S(h_1),h_3,g),
\end{eqnarray*}
which immediately implies that the above expression reduces to
$0.$
Employing an argument analogous to the one used in the proof of \cite[Theorem 4.12]{Yu26}, we obtain from the definition of multiplication that any element of $M^{\Box n}$ verifies (\ref{eq:alpha,beta}) for all $n \ge 1$.
This completes the proof.
\end{proof}

\begin{remark}\rm
\begin{itemize}
\item[(1)]The coquasi-Hopf algebra structure of $\operatorname{CoT}_H(M)$ has already been considered in some special cases, by specializing $H$ appropriately. Huang \cite{Hua09} generalized the theory of Hopf quivers (\cite{CR02}) to the coquasi-Hopf algebra case, where $H$ is taken to be a pointed cosemisimple coquasi-Hopf algebra of the form $(\k G,\omega)$, with $\omega$ a normalized $3$-cocycle on a group $G$. Subsequently, in our previous work \cite{Yu26}, we introduced generalized Hopf quivers, in which $H$ is taken to be a (not necessarily pointed) cosemisimple coquasi-Hopf algebra.
\item[(2)]The coquasi-Hopf algebra structure on $\operatorname{CoT}_H(M)$ constructed in Proposition \ref{prop:CotCQH} is determined by the Majid bimodule structure on $M$, or equivalently, by Lemma \ref{lem:Majidmod=YD}, the Yetter-Drinfeld module structure on $M^{co H}$. In other words, for the same vector space $M$, endowing it with two distinct Majid bimodule structures necessarily gives rise to distinct multiplication structures on $\operatorname{CoT}_H(M)$. In what follows, whenever we refer to the coquasi-Hopf algebra structure on $\operatorname{CoT}_H(M)$, where $H$ is a coquasi-Hopf algebra and $M$ is an $H$-Majid bimodule, we always mean the one constructed in the proof of Proposition \ref{prop:CotCQH}. Moreover, it should be pointed out that Nichols' construction concerning the antipode does not apply in this setting.
\end{itemize}
\end{remark}
Bosonization is well established for Hopf algebras, and was later extended to coquasi-Hopf algebras in \cite{AP13}. In particular, $\operatorname{CoT}_H(M)$ can be realized as such a bosonization.
Let $H$ be a coquasi-Hopf algebra over $\k$, $M\in {}_H^{H}\mathscr{M}^{H}_H$ and $V=M^{co H}$. Lemma \ref{lem:Majidmod=YD} gives $H^{co H}=\k1$ and $(M^{\Box n})^{co H}=V^{\otimes n}$, which immediately implies that $(\operatorname{CoT}_H(M))^{co H}=T(V)$. Note that there exist morphisms of coquasi-Hopf algebras
$$
\pi: \operatorname{CoT}_H(M) \rightarrow H
\;\;\text{and}\;\;
\iota: H \rightarrow \operatorname{CoT}_H(M),
$$
where $\pi$ is the natural projection onto the degree $0$ component of $\operatorname{CoT}_H(M)$, and $\iota$ is the canonical embedding, satisfying
$
\pi \iota = \operatorname{id}_H.
$
Since every coquasi-Hopf algebra admits a preantipode by \cite[Theorem 3.10]{AP12}, it follows from \cite[Theorem 5.8]{AP13} that
\begin{equation*}
  \begin{array}{rcl}
\psi: T(V) \# H &\rightarrow& \operatorname{CoT}_H(M)\\
        r\# h\;\; &\mapsto & \;\;r \iota(h)
\end{array}
\end{equation*}
is an isomorphism of coquasi-bialgebras. Indeed, we can endow $T(V)\#H$ with a coquasi-antipode structure so that the above isomorphism becomes a coquasi-Hopf algebra isomorphism. This result generalizes \cite[Theorem 11]{Ros98} from Hopf algebras to coquasi-Hopf algebras, and parallels \cite[Theorem 5.11]{Bul20} in the quasi-Hopf setting. It is worth noting that, in this case, $T(V)$ is a Hopf algebra in ${}^H_H\mathcal{YD}$, whose braided Hopf algebra structure is given in \cite[Section 5]{AP13}, although it need not be associative as an ordinary algebra.

Moreover, let $S_H(M)$ be the coquasi-Hopf subalgebra of the coquasi-Hopf algebra $\operatorname{CoT}_H(M)$ generated by $H$ and $M$. Following \cite[Subsection 1.4]{Ros98}, we refer to $S_H(M)$ as the \textit{quasi-quantum symmetric algebra}. By Lemma \ref{lem:Majidmod=YD}, $V$ is a left-left Yetter-Drinfeld module over $H$ if and only if $M$ is an $H$-Majid bimodule, and in this case
$$
S_H(M)\cong \mathcal{B}(V)\# H,
$$
where $\mathcal{B}(V)$ is the Nichols algebra of $V$ (see \cite[Subsection 5.7]{BB13} for the definition). In \cite{HLY11}, Huang, Liu, and Ye constructed a coquasi-Hopf algebra structure on the path coalgebra of a basic cycle. This is an (infinite-dimensional) cotensor coquasi-Hopf algebra $\operatorname{CoT}_H(M)$, where $H$ is $(\k \mathbb{Z}_n, \omega)$, and $M$ is the vector space spanned by the arrows. The subalgebra $M(n,s,q)$ (\cite[Proposition 3.10]{HLY11}) generated by the vertices and arrows is exactly a finite-dimensional quasi-quantum symmetric algebra.

\section{Coquasi-Hopf algebras without the dual Chevalley property}\label{section3}
\subsection{Graded coquasi-Hopf algebras from the standard filtration}
In this subsection, we aim to study coquasi-Hopf algebras without the dual Chevalley property. Our first goal is to extend the ideas of Andruskiewitsch and Cuadra \cite{AC13} to the coquasi-Hopf algebra setting.

We let $H$ be a coquasi-Hopf algebra over $\k,$ and let $\{H_n\}_{n\geq0}$ denote its coradical filtration. Inspired by \cite{AC13}, we define the \textit{standard filtration} $\{H_{[n]}\}_{n\geq 0}$ of $H$ as follows:
\begin{itemize}
    \item [(1)] The Majid coradical of $H$ is the coquasi-Hopf subalgebra $H_{[0]}$ of $H$ generated by $H_0;$
    \item[(2)] The subsequent terms are defined recursively by $$H_{[n+1]}:=H_{[0]}\wedge H_{[n]}=\Delta^{-1}(H_{[0]}\otimes H+ H\otimes H_{[n]}),\;\; n\geq 0.$$
\end{itemize}
Note that if $H$ has the dual Chevalley property, i.e., its coradical is a coquasi-Hopf subalgebra, then $H_{[0]}=H_0$ and the standard filtration coincides with the coradical filtration.

With the notations above, we have
\begin{lemma}\label{lem:Majidcoradical}
Let $H$ be a coquasi-Hopf algebra over $\k$ with a coquasi-antipode $(S,\alpha,\beta)$ such that $S(H_0)\subseteq H_0$. Then
\begin{itemize}
    \item [(1)] $H_{[0]}$ is a coquasi-Hopf subalgebra of $H$ and its coradical is $H_0$;
    \item[(2)] $H_n \subseteq H_{[n]}$ and $\{H_{[n]}\}_{n\geq 0}$ is a coquasi-Hopf algebra filtration.
\end{itemize}
\end{lemma}
\begin{proof}
\begin{itemize}
    \item [(1)] Similar to the proof of \cite[Lemma 1.1(i)]{AC13}, we can show that $H_{[0]}$ is closed under multiplication and comultiplication, and that $(H_{[0]})_0=H_0$. By (\ref{eq:S(hg)}), we have $S(H_{[0]})\subseteq H_{[0]}$ and thus $H_{[0]}$ is a coquasi-Hopf subalgebra.
    \item [(2)]The proof is entirely analogous to that of \cite[Lemma 1.1(ii)]{AC13}.
\end{itemize}
\end{proof}
\begin{remark}\rm\label{rm:S(H0)<H0}
We remark that if the coquasi-antipode $(S,\alpha,\beta)$ of $H$ is injective, then $S$ maps a simple subcoalgebra to a simple subcoalgebra, so that $S(H_0) \subseteq H_0$. According to \cite[Corollary 4.21]{Yu26}, the coquasi-antipode of a coquasi-Hopf algebra with the dual Chevalley property is injective.
\end{remark}
Analogous to the associated graded Hopf algebra structure arising from the filtration of a Hopf algebra, we have the following lemma, which generalizes \cite[Proposition 6.2]{AP13}, where $H_0$ is a coquasi-Hopf subalgebra.

\begin{lemma}\label{lem:gr}
Let $(H, m, u, \Delta, \varepsilon, \Phi)$ be a coquasi-Hopf algebra over $\k$ with a coquasi-antipode $(S,\alpha,\beta)$ such that $S(H_0)\subseteq H_0$. Then the associated graded coalgebra
$$
\operatorname{gr} H:=\bigoplus_{n\geq0}H_{[n]}/H_{[n-1]}
$$
is a graded coquasi-Hopf algebra. Its structure maps are given on homogeneous elements by
\begin{equation*}
m_{\operatorname{gr} H}(\bar{x} \otimes \bar{y}):=xy+H_{[m+n-1]}, \;\;
1_{\operatorname{gr} H}:=1_H,
\end{equation*}
\begin{equation*}
\Delta_{\operatorname{gr} H}(\bar{x}) :=\sum\limits_{i+j=m} (x_1^{(i)} + H_{[i-1]}) \otimes (x_2^{(j)} + H_{[j-1]}),
\;\;
\varepsilon_{\operatorname{gr} H}(\bar{x})
:=\delta_{m,0}\varepsilon(x),
\end{equation*}
\begin{equation*}
\Phi_{\operatorname{gr} H}(\bar{x},\bar{y},\bar{z})
:=\delta_{m,0}\delta_{n,0}\delta_{p,0}\Phi(x,y,z),
\end{equation*}
\begin{equation*}
S_{\operatorname{gr} H}(\bar{x})
:=S(x)+H_{[m-1]},\;\;
\alpha_{\operatorname{gr} H}(\bar{x})
:=\delta_{m,0}\alpha(x),\;\;
\beta_{\operatorname{gr} H}(\bar{x})
:=\delta_{m,0}\beta(x),
\end{equation*}
where $\bar{x}=x+H_{[m-1]}$, $\bar{y}=y+H_{[n-1]}$, and $\bar{z}=z+H_{[p-1]}$, and $\Delta(x)=\sum_{i+j=m} x_1^{(i)}\otimes x_2^{(j)} \subseteq \sum_{i+j=m} H_{[i]}\otimes H_{[j]}$. Moreover, we have $$(\operatorname{gr} H)_{[n]}=\bigoplus_{i\leq n} \operatorname{gr} H (i).$$
\end{lemma}
\begin{proof}
The proof that $\operatorname{gr} H:=\bigoplus_{n\geq0}H_{[n]}/H_{[n-1]}$ is a graded coquasi-Hopf algebra is analogous to that in the case of coquasi-Hopf algebras with the dual Chevalley property; see \cite[Proposition 6.2]{AP13}. Moreover, proceeding as in the proof of \cite[Proposition 1.2]{AC13}, we obtain that $$(\operatorname{gr} H)_{[n]}=\bigoplus_{i\leq n} \operatorname{gr} H (i).$$
\end{proof}

\begin{corollary}\label{coro:grbosonization}
With the notations of Lemma \ref{lem:gr}, we have $$\operatorname{gr} H \cong (\operatorname{gr} H)^{co H_{[0]}} \# H_{[0]}.$$
\end{corollary}
\begin{proof}
Since every coquasi-Hopf algebra admits a preantipode by \cite[Theorem 3.10]{AP12}, the result follows immediately from \cite[Theorem 5.8]{AP13}.
\end{proof}

\begin{remark}\rm
In \cite[Questions I, II, III]{AC13}, Andruskiewitsch and Cuadra outlined a three-step program for the classification of Hopf algebras with injective antipode. This approach is also discussed in Angiono's report at the 2026 International Congress of Mathematicians \cite[Subsection 3.2]{Ang26}. The same strategy can be applied to the classification of coquasi-Hopf algebras:
\begin{itemize}
    \item Classify all coquasi-Hopf algebras $L$ generated by a given cosemisimple coalgebra $C$, with an injective anti-coalgebra map $S: C \to C$ such that $S_L|_C = S$.
    \item Classify all connected graded Hopf algebras $R$ in ${}^L_L \mathcal{YD}$;
    \item Classify all liftings (or deformations) $H$ satisfying $\operatorname{gr} H\cong R\# L.$
\end{itemize}
It is worth noting that, in contrast to the Hopf algebra setting, the first step of the program involves a much wider variety of coquasi-Hopf algebras. We claim that if $H_{[0]}$ is genuine, i.e., not twist equivalent to a Hopf algebra, then $\operatorname{gr}H$ is also genuine. Indeed, the proof is analogous to that of \cite[Corollary 4.19]{Yu26}. Furthermore, in the third step, the computation of liftings becomes more complicated, as one must determine all possible reassociators from those concentrated in degree $0$.
\end{remark}

Next we construct an example of a coquasi-Hopf algebra without the dual Chevalley property.
Inspired by Question 3.3 in Angiono's report at the 2026 International Congress of Mathematicians \cite{Ang26}, we consider, for instance, certain constructions starting from the dual of the $8$-dimensional Radford Hopf algebra $R_{2,2}$ (\cite{Rad75}), and then apply bosonization to construct examples of coquasi-Hopf algebras. We note that the dual of the Radford Hopf algebra can be found in \cite{ACE15}. More precisely, we construct a coquasi-Hopf algebra of the form
$$
\Bbbk\langle y \rangle/(y^2)\#\bigl(R_{2,2}^*\otimes \Bbbk(\mathbb{Z}_2\times \mathbb{Z}_2,\omega)\bigr),
$$
which does not possess the dual Chevalley property, and whose structure is given explicitly as follows.
\begin{example}\rm\label{ex:M(A,X,z,g,y)}
Let $\k$ be an algebraically closed field of characteristic $0$. Let $M(A,X,z,g,y)$ be the $64$-dimensional coquasi-Hopf algebra that is generated as a quasi-algebra by $A, X, z, g, y$, subject to the relations
\begin{eqnarray*}
&A^4=1,\;\; X^2=0,\;\; AX=\sqrt{-1}XA,&\\
&z^2=g^2=1,\;\;zg=gz,\;\; Az=zA,\;\; Xz=zX,\;\; Ag=gA,\;\; Xg=gX,&\\
& y^2=0,\;\; Ay=yA,\;\; Xy=yX,\;\;  gy=-yg, \;\; zy=yz.&
\end{eqnarray*}
Its coalgebra structure, reassociator, and coquasi-antipode are determined respectively by
\begin{eqnarray*}
&&\Delta(A)= A\otimes A+2XA^3\otimes XA,\;\; \varepsilon(A)=1,\;\;S(A)=A^3,\\
&&\Delta(X)=X\otimes1+A^2\otimes X,\;\; \varepsilon(X)=0,\;\;S(X)=XA^2,\\
&&\Delta(z)=z\otimes z,\;\; \Delta(g)=g\otimes g,\;\;\varepsilon(z)=\varepsilon(g)=1,\;\;S(z)=z,\;\;S(g)=g,\\
&&\Delta(y)=y\otimes1+g\otimes y,\;\;\varepsilon(y)=0,\;\;S(y)=-gy,\\
&&\alpha(A^{r_A} X^{r_X} z^{r_z} g^{r_g} y^{r_y})=\delta_{r_X,0} \delta_{r_y,0},\;\; \beta(A^{r_A} X^{r_X} z^{r_z} g^{r_g} y^{r_y})=(-1)^{r_z} \delta_{r_X,0} \delta_{r_y,0},\\
&&\Phi(A^{r_A} X^{r_X} z^{r_z} g^{r_g} y^{r_y}, A^{s_A} X^{s_X} z^{s_z} g^{s_g} y^{s_y}, A^{t_A} X^{t_X} z^{t_z} g^{t_g} y^{t_y})\\
&&=(-1)^{r_z s_z t_z}\delta_{r_X, 0}\delta_{r_y,0}\delta_{s_X, 0}\delta_{s_y,0}\delta_{t_X, 0}\delta_{t_y,0}.
\end{eqnarray*}
In fact, disregarding the reassociator and the coquasi-antipode, the underlying structure of $M(A,X,z,g,y)$ is already a bialgebra. We know that
\begin{eqnarray*}
&&M(A,X,z,g,y)_0\\
&=&\k 1\oplus \k A^2\oplus \k z\oplus \k g \oplus \k A^2 z \oplus \k A^2 g \oplus \k zg\oplus \k A^2 zg \oplus C_1 \oplus C_z \oplus C_g \oplus C_{zg},
\end{eqnarray*}
where the comatrix coalgebras corresponding to $C_1$, $C_z$, $C_g$ and $C_{zg}$ are respectively represented in matrix form as
\begin{eqnarray*}
&\C_1=\begin{pmatrix}
A & \sqrt{2} XA^3 \\
\sqrt{2} XA & A^3
\end{pmatrix},\;\;
\C_z=\begin{pmatrix}
A z& \sqrt{2} XA^3 z\\
\sqrt{2} XAz & A^3z
\end{pmatrix},
\\
&\C_g=\begin{pmatrix}
A g& \sqrt{2} XA^3g \\
\sqrt{2} XAg & A^3g
\end{pmatrix},\;\;
\C_{zg}=\begin{pmatrix}
A zg& \sqrt{2} XA^3 zg\\
\sqrt{2} XA zg& A^3zg
\end{pmatrix},
\end{eqnarray*}
and $C_1$, $C_z$, $C_g$ and $C_{zg}$ are precisely the linear spans of the matrix entries of the corresponding matrices. It follows that $$\dim_{\k} (M(A,X,z,g,y)_0)=24.$$ Evidently, this coquasi-Hopf algebra fails to satisfy the dual Chevalley property; indeed, we have $$M(A,X,z,g,y)_{[0]}=\span\{A^{r_A} X^{r_X} z^{r_z} g^{r_g} \mid 0 \leq r_A \leq 3, 0\leq r_X, r_z, r_g\leq 1\},$$ and therefore $\dim_{\k} M(A,X,z,g,y)_{[0]}=32.$ Moreover, we have $$M(A,X,z,g,y)= M(A,X,z,g,y)_{[1]},$$ where $$M(A,X,z,g,y)_{[1]}/M(A,X,z,g,y)_{[0]}=\span\{A^{r_A} X^{r_X} z^{r_z} g^{r_g} y\mid 0\leq r_A \leq 3, 0\leq r_X, r_z, r_g\leq 1\}.$$
\end{example}

We now consider the relationship between the graded coquasi-Hopf algebra $\operatorname{gr} H$ constructed in Lemma \ref{lem:gr} and the cotensor coquasi-Hopf algebra.
\begin{theorem}\label{thm:embedding}
Let $H$ be a coquasi-Hopf algebra over $\k$ with a coquasi-antipode $(S,\alpha,\beta)$ such that $S(H_0)\subseteq H_0$, and let $\operatorname{gr} H$ denote the graded coquasi-Hopf algebra constructed in Lemma \ref{lem:gr}. Then $\operatorname{CoT}_{H_{[0]}}(H_{[1]}/H_{[0]})$ admits a graded coquasi-Hopf algebra structure, and there is a graded coquasi-Hopf algebra embedding $\psi:\operatorname{gr} H\hookrightarrow \operatorname{CoT}_{H_{[0]}}(H_{[1]}/H_{[0]})$ such that $\psi\mid_{H_{[0]}\oplus H_{[1]}/H_{[0]}}=\id$.
\end{theorem}
\begin{proof}
By Lemma \ref{lem:gradedMajid}, $H_{[1]}/H_{[0]}$ admits an $H_{[0]}$-Majid bimodule structure. Then, by Proposition \ref{prop:CotCQH}, the cotensor coquasi-Hopf algebra $\operatorname{CoT}_{H_{[0]}}(H_{[1]}/H_{[0]})$ is a graded coquasi-Hopf algebra, with the structure defined in the proof of Proposition \ref{prop:CotCQH}.
We first construct a coalgebra map
$$
\psi:\operatorname{gr} H\hookrightarrow \operatorname{CoT}_{H_{[0]}}(H_{[1]}/H_{[0]}).
$$
According to Lemma \ref{lemma:universal}, it suffices to construct a coalgebra map $\psi_0:\operatorname{gr} H\to H_{[0]}$ and a $H_{[0]}$-bicomodule map $\psi_1:\operatorname{gr} H\to H_{[1]}/H_{[0]}$ such that for each $x\in \operatorname{gr} H$, only finitely many $\psi_i(x)$ are nonzero, where $\psi_i$ is defined as in Lemma \ref{lemma:universal}. We take $\psi_0$ and $\psi_1$ to be the canonical projections
$$
\psi_0:\operatorname{gr}H\rightarrow H_{[0]}, \;\;
\psi_1:\operatorname{gr}H\rightarrow H_{[1]}/H_{[0]}.
$$
Clearly, $\psi_n(H_{[m]}/H_{[m-1]})=0$ for $n\neq m$. Hence
$$
\psi=\sum_i \psi_i:\operatorname{gr}H\rightarrow \operatorname{CoT}_{H_{[0]}}(H_{[1]}/H_{[0]})
$$
is a graded coalgebra map. Lemma \ref{lem:gr} yields
$$
(\operatorname{gr} H)_{[n]}=\bigoplus_{i\leq n} \operatorname{gr} H(i).
$$
By Lemma \ref{lem:Majidcoradical}, we have $(\operatorname{gr} H)_n\subseteq (\operatorname{gr} H)_{[n]}$, which implies that $\psi$ restricted to $(\operatorname{gr} H)_1$ is injective. Consequently, $\psi$ is injective by \cite[Theorem 5.3.1]{Mon93}. To establish that $\psi$ is a morphism of coquasi-Hopf algebras, we first note that Lemma \ref{lemma:universal} implies $\psi m_{\operatorname{gr} H}= m^\prime(\psi\otimes \psi)$ and $\psi S_{\operatorname{gr} H}=S^\prime\psi$. The compatibility with $\alpha^\prime$, $\beta^\prime$, and $\Phi^\prime$ is a direct consequence of their construction. This completes the proof.
\end{proof}
\subsection{The bijectivity of the coquasi-antipode}
Let $H$ be a coquasi-Hopf algebra over $\k$ with bijective coquasi-antipode $S$. The category $\mathscr{M}^H$ of finite-dimensional right $H$-comodules is a tensor category. For $V \in \mathscr{M}^H$, its left dual $V^*$ is the dual vector space endowed with the right $H$-coaction $\rho^R_{V^*}(f)=f_0\otimes f_1$ determined by
$$
\langle f_0,v\rangle f_1=\langle f,v_0\rangle S(v_1)
$$
for all $v\in V$ and $f\in V^*$. The right dual of $V$ is defined similarly by replacing $S$ with $S^{-1}$ in the formula above. Hence, we aim to determine whether the coquasi-antipode of a coquasi-Hopf algebra is bijective.

We first show that if the coquasi-antipode of $H$ is bijective, then so is that of $\operatorname{CoT}_H(M)$.
\begin{lemma}\label{lem:antipodebijective}
Let $H$ be a coquasi-Hopf algebra over $\k$ with a bijective coquasi-antipode $(S, \alpha, \beta)$, and let $M$ be an $H$-Majid bimodule. Then the coquasi-antipode $(S^\prime, \alpha^\prime, \beta^\prime)$ of $\operatorname{CoT}_H(M)$ constructed in the proof of Proposition \ref{prop:CotCQH} is bijective.
\end{lemma}
\begin{proof}
For our purposes, we now construct a morphism
$$S^{\prime\prime}: (\operatorname{CoT}_H(M))^{cop} \rightarrow \operatorname{CoT}_H(M)$$
satisfying
$S^{\prime\prime}S^\prime = S^\prime S^{\prime\prime} = \mathrm{id}$.
By \cite[Proposition 2.16]{Yu26}, we have
$$
 M{}^{co H} \otimes H\cong M,\;\; m\otimes h\mapsto m\cdot h,
$$
where $M{}^{co H}= \{m\in M\mid \rho_M^R(m)=m\otimes 1_H\}$. Following the same argument as in Proposition \ref{prop:CotCQH}, we first specify the action of $T^\prime$ on elements $n\cdot h$ with $h\in H$ and $n\in M^{co H}$ by
$$
T^\prime(n\cdot h):=-f^{-1}(S^{-1}(h_1), S^{-1}(n_{-3})) \alpha(S^{-1}(n_{-1})) S^{-1}(h_2)\cdot (n_{0} \cdot S^{-1}(n_{-2})).
$$
Let $X=(\operatorname{CoT}_H(M))^{cop}$, and define $S^{\prime\prime}_0$ as the composition
$$
X\xrightarrow{\pi_0} H\xrightarrow{S^{-1}}H.
$$
We denote the composition
$$
X\xrightarrow{\pi_1}M\xrightarrow{T^\prime}M
$$
by $S^{\prime\prime}_1$.
For $n\geq 2, $ define $S^{\prime\prime}_n$ as the composition
$$
(\operatorname{CoT}_H(M))^{cop}\xrightarrow{\Delta^{n-1}} ((\operatorname{CoT}_H(M))^{cop})^{\otimes n} \xrightarrow{(S^{\prime\prime}_1)^{\otimes n}} M^{\otimes n}.
$$
By Lemma \ref{lemma:universal}, one obtains a coalgebra map
$$
S^{\prime\prime} = \sum_{i=0} S^{\prime\prime}_i: (\operatorname{CoT}_H(M))^{cop}\rightarrow \operatorname{CoT}_H(M).
$$
A similar argument as in the proof of \cite[Proposition 4.20]{Yu26} shows that $$S^{\prime\prime}S^\prime=S^\prime S^{\prime\prime}=\operatorname{id}.$$
\end{proof}

\begin{lemma}\label{lem:H[0]Sinj}
Let $H$ be a coquasi-Hopf algebra over $\k$ with a coquasi-antipode $(S,\alpha,\beta)$ such that $S(H_0)\subseteq H_0$. If $S\mid_{H_{[0]}}$ is bijective, then $S$ is injective.
\end{lemma}
\begin{proof}
According to Theorem \ref{thm:embedding}, there is a graded coquasi-Hopf algebra embedding $\psi:\operatorname{gr} H\hookrightarrow \operatorname{CoT}_{H_{[0]}}(H_{[1]}/H_{[0]})$ such that $\psi\mid_{H_{[0]}\oplus H_{[1]}/H_{[0]}}=\id$. Then by Lemma \ref{lem:antipodebijective}, the coquasi-antipode $(S^\prime, \alpha^\prime, \beta^\prime)$ of $\operatorname{CoT}_{H_{[0]}}(H_{[1]}/H_{[0]})$ constructed in the proof of Proposition \ref{prop:CotCQH} is bijective. This means that $\operatorname{gr}H$ admits a coquasi-antipode $S_{\operatorname{gr} H}$ that is bijective on $H_{[0]}\oplus H_{[1]}/H_{[0]}$. Lemma \ref{lem:Majidcoradical} gives $(\operatorname{gr} H)_n\subseteq (\operatorname{gr} H)_{[n]}$; therefore the restriction of $S_{\operatorname{gr} H}$ to $(\operatorname{gr} H)_1$ is injective. It follows from \cite[Theorem 5.3.1]{Mon93} that $S_{\operatorname{gr} H}$ is injective. Consequently, $S$ is injective.
\end{proof}

Recall that a \textit{ring category} over $\k$ is a locally finite $\k$-linear abelian monoidal category $\mathcal{C}$ with a bilinear and biexact tensor product such that $\operatorname{End}_{\mathcal{C}}(\mathbbm{1})\cong \k$. Here, by a \textit{locally finite category} \cite[Definition 1.8.1]{EGNO15}, we mean one whose morphism spaces are finite-dimensional and in which every object has finite length. A locally finite category is said to be \textit{finite} \cite[Definition 1.8.6]{EGNO15} if it has enough projective objects and finitely many isomorphism classes of simple objects. It is clear that the category of finite-dimensional comodules of a coquasi-Hopf algebra ( not necessarily finite-dimensional) forms a ring category whose objects all possess left duals.

Let $\mathscr{C}$ be a ring category, and let
$$
0 \to X \xrightarrow{\iota} Z \xrightarrow{\pi} Y \to 0
$$
be a short exact sequence in $\mathscr{C}$. Etingof and Penneys \cite[Lemma 4.2]{EP26} proved that, assuming $\mathscr{C}$ is braided, if $X$ and $Y$ have both left and right duals, then so does $Z$. They also noted that the use of the braiding in the proof of \cite[Lemma 4.2]{EP26} is not essential (see \cite[Remark 4.4]{EP26}). For details on braided categories, see \cite{EGNO15}. In \cite[Appendix A]{SY26}, Shimizu and Yadav showed that \cite[Lemma 4.2]{EP26} remains valid in the above form without assuming a braiding, by means of a careful examination of the original proof. Later, they \cite[Theorem A.2]{SY26} proved that if $\mathscr{C}$ is a ring category in which every object has finite length and every simple object of $\mathscr{C}$ has both left and right duals, then so does every object of $\mathscr{C}$.

From the above analysis, we obtain the following proposition.
\begin{theorem}\label{thm:H[0]Sbijiective}
Let $H$ be a coquasi-Hopf algebra over $\k$ with a coquasi-antipode $(S,\alpha,\beta)$ such that $S(H_0)\subseteq H_0$. Then $S$ is bijective if and only if $S\mid_{H_{[0]}}$ is bijective.
\end{theorem}
\begin{proof}
Since the ``only if'' part is trivial, it remains only to prove the ``if'' part.
By Lemma \ref{lem:H[0]Sinj}, it suffices to prove that $S$ is surjective. Note that the category $\mathscr{M}^H$ of finite-dimensional right $H$-comodules is a ring category in which every object has a left dual. We claim that every object of $\mathscr{M}^H$ has a right dual. In fact, since $S\mid_{H_{[0]}}$ is bijective, it follows that $\mathscr{M}^{H_{[0]}}$ is a tensor category. By Lemma \ref{lem:Majidcoradical}(2), we have $H_0\subseteq H_{[0]}$; consequently, every simple right $H$-comodule has a right dual. According to \cite[Theorem A.2]{SY26}, every finite-dimensional right $H$-comodule has a right dual. If we regard every finite-dimensional subcoalgebra $C$ of $H$ as a right $H$-comodule via the comultiplication, then, by \cite[Remark 2.10.3]{EGNO15}, we have
$$
({}^*C)^* \cong C.
$$
This implies that $C\subseteq S(H)$. By the fundamental theorem of coalgebras \cite[Theorem 1.4.7]{DNR01}, every element of $H$ is contained in a finite-dimensional subcoalgebra. Thus $S$ is surjective; together with its injectivity, this implies that $S$ is bijective.
\end{proof}
According to \cite[Corollary 4.4]{BC03}, a coquasi-antipode of a finite-dimensional coquasi-Hopf algebra is bijective. Thus we have
\begin{corollary}
Let $H$ be a coquasi-Hopf algebra over $\k$ with a coquasi-antipode $(S,\alpha,\beta)$ such that $S(H_0)\subseteq H_0$. If $H_{[0]}$ is finite-dimensional, then $S$ is bijective.
\end{corollary}
\begin{corollary}
Let $H$ be a coquasi-Hopf algebra over $\k$, and let $H^\prime$ be a coquasi-Hopf subalgebra of $H$ containing $H_0$. If $S\mid_{H^\prime}$ is bijective, then $S$ is bijective.
\end{corollary}
It is well known that the antipode of a Hopf algebra with the dual Chevalley property is necessarily bijective; see \cite[Corollary 3.6]{Rad77}. However, the proof of \cite[Corollary 3.6]{Rad77} cannot be carried over to the coquasi-Hopf algebra setting. In contrast, Theorem \ref{thm:H[0]Sbijiective} enable us to prove that the coquasi-antipode of a coquasi-Hopf algebra with the dual Chevalley property is bijective.
\begin{proposition}\label{prop:Chevalleybijective}
 Let $H$ be a (not necessarily finite-dimensional) coquasi-Hopf algebra over $\k$ with the dual Chevalley property, and let $(S,\alpha,\beta)$ be a coquasi-antipode. Then $S$ is bijective. In particular, the coquasi-antipode of a pointed coquasi-Hopf algebra is bijective.
\end{proposition}
\begin{proof}
According to Theorem \ref{thm:H[0]Sbijiective}, it is enough to show that $S\mid_{ H_0}$ is bijective. This, however, is a direct consequence of \cite[Lemma 2.7]{Yu26}.
\end{proof}
We remark that the assumption that the field is algebraically closed, made in \cite{Yu26}, is not actually needed in the proof of \cite[Lemma 2.7]{Yu26}.

Recall that a tensor category $\mathscr{C}$ is said to have the \emph{Chevalley property}, if the tensor product of every two simple objects of $\mathscr{C}$ is semisimple (\cite[Definition 4.12.3]{EGNO15}). As a corollary of Proposition \ref{prop:Chevalleybijective}, we have
\begin{corollary}
Let $H$ be a (not necessarily finite-dimensional) coquasi-Hopf algebra over $\k$. Then $H$ has the dual Chevalley property if and only if $\mathscr{M}^H$ is a tensor category with the Chevalley property.
\end{corollary}
\begin{proof}
The ``if'' part follows as a corollary of \cite[Proposition 4.2.9]{EGNO15}, whereas the ``only if'' part is obtained from Proposition \ref{prop:Chevalleybijective}.
\end{proof}

\section{Coquasitriangular structures}\label{section4}
\subsection{Coquasitriangular structures on cotensor coquasi-Hopf algebras}
Recall that a \textit{coquasitriangular} coquasi-Hopf algebra (see \cite[Definition 11.14]{BCPV19}) is a pair $(H, R)$, where $H$ is a coquasi-Hopf algebra and $R: H \otimes H \to \k$ is a convolution-invertible linear map satisfying the following conditions:
\begin{eqnarray}
&R(a_1, b_1)a_2b_2=b_1a_1 R(a_2, b_2),&\label{eq:CQT1}\\
&R(ab, c)=\Phi(c_1,a_1,b_1)R(a_2,c_2)\Phi^{-1}(a_3,c_3,b_2)R(b_3,c_4)\Phi(a_4,b_4,c_5),&\label{eq:CQT2}\\
&R(a,bc)=\Phi^{-1}(b_1,c_1,a_1)R(a_2,c_2)\Phi(b_2,a_3,c_3)R(a_4,b_3)\Phi^{-1}(a_5,b_4,c_4).&\label{eq:CQT3}
\end{eqnarray}
If, in addition,
\begin{eqnarray}\label{eq:CQT4}
 R(a_{1}, b_{1}) R(b_{2}, a_{2})=\varepsilon (a) \varepsilon (b),
\end{eqnarray}
then $(H, R)$ is called \textit{cotriangular}.

We now investigate when a coquasitriangular structure exists on $\operatorname{CoT}_H(M)$ and provide characterizations of its coquasitriangular structures.
\begin{proposition}\label{prop:CotHMCQT}
Let $(H, R)$ be a coquasitriangular coquasi-Hopf algebra over $\k$, and let $V$ be a left $H$-comodule. Set $M = V \otimes H$. Then $M$ admits an $H$-Majid bimodule structure, and consequently $\operatorname{CoT}_H(M)$ admits a coquasi-Hopf algebra structure. Moreover, if $H$ is cotriangular, then $\operatorname{CoT}_H(M)$ admits a cotriangular coquasi-Hopf algebra structure.
\end{proposition}
\begin{proof}
The proof is analogous to that of \cite[Lemma 5.28]{Yu26}, but we include it for the sake of completeness.
According to \cite[Example 2.22]{Yu26}, we can show that $V$ admits a left-left Yetter-Drinfeld module structure via
$ h\vartriangleright v:=R(v_{-1}, h)v_0.$ It then follows from Lemma \ref{lem:Majidmod=YD} that $M$ admits an $H$-Majid bimodule structure, given explicitly as follows:
\begin{eqnarray*}
&l\cdot
(v\otimes h):=\Phi (l_{1}, v_{-3}, h_{1})\Phi
^{-1}(v_{-1}, l_{3},
h_{2})R(v_{-2},l_2) v_0\otimes l_{4}h_{3}, & \\
    &(v\otimes
h)\cdot l:=\Phi ^{-1}(v_{-1}, h_{1}, l_{1})v_{0}\otimes
h_{2}l_{2}, &  \\
&\rho _{V\otimes H}^{L}\left( v\otimes h\right) :=v_{-1}h_{1}\otimes
(v_{0}\otimes h_{2}), & \\
&\rho _{V\otimes H}^{R}\left( v\otimes h\right) :=(v\otimes h_{1})\otimes
h_{2}.&
\end{eqnarray*}
Hence, by Proposition \ref{prop:CotCQH}, $\operatorname{CoT}_H(M)$ admits a coquasi-Hopf algebra structure. We now extend $R$ trivially to be a function $R^\prime$ on $\operatorname{CoT}_H(M) \otimes \operatorname{CoT}_H(M)$ by defining it on homogeneous elements as follows: $$R^\prime(a,b):=\left\{
\begin{aligned}
R(a,b), &\text{ if} ~~~ a,b\in H; \\
0,\;\;\;\;&\text{ otherwise}.
\end{aligned}
\right.$$
Now we claim that $(\operatorname{CoT}_H(M), R^\prime)$ is cotriangular. Indeed, the convolution invertibility of $R^\prime$, as well as identities (\ref{eq:CQT2}), (\ref{eq:CQT3}) and (\ref{eq:CQT4}), follow immediately from the definition of $R^\prime$. It remains to verify that $R^\prime $ satisfies (\ref{eq:CQT1}). For any $l\in H$ and $v\otimes h \in M$, by (\ref{eq:CQT1})-(\ref{eq:CQT4}), we have
\begin{eqnarray*}
&&R(l_1,v_{-1}h_1) l_2\cdot (v_0\otimes h_2)\\
&=&\Phi^{-1}(v_{-7},h_1,l_1)R(l_2,h_2)\Phi(v_{-6},l_3,h_3)R(l_4,v_{-5})\Phi^{-1}(l_5,v_{-4},h_4)\\
&&\Phi(l_6,v_{-3},h_5)\Phi^{-1}(v_{-1},l_8,h_6)R(v_{-2},l_7) v_0\otimes R^{-1}(l_9,h_7)h_8l_{10} R(l_{11},h_9)\\
&=&\Phi^{-1}(v_{-1},h_1,l_1)v_0\otimes h_2l_{2} R(l_{3},h_3)\\
&=&(v\otimes h_1) \cdot l_1 R(l_2, h_2),
\end{eqnarray*}
and
\begin{eqnarray*}
&&R(v_{-1}h_1, l_1)(v_0\otimes h_2)\cdot l_2    \\
&=& \Phi(l_1, v_{-5}, h_1) R(v_{-4}, l_2)\Phi^{-1}(v_{-3}, l_3, h_2) R(h_3, l_4)
\Phi(v_{-2}, h_4, l_5) \Phi^{-1}(v_{-1}, h_5, l_6) \\&& v_0\otimes R^{-1}(h_6, l_7)l_8 h_7 R(h_8, l_9)\\
&=&\Phi(l_1, v_{-3},h_1)R(v_{-2},l_2)\Phi^{-1}(v_{-1},l_3,h_2)v_0\otimes l_4h_3 R(h_4, l_5)\\
&=& l_1 \cdot (v\otimes h_1) R(h_2, l_2).
\end{eqnarray*}
This means that $R^\prime$ satisfies \eqref{eq:CQT1} when restricted to
$$
(H \otimes M) \oplus (M \otimes H) \subset \operatorname{CoT}_H(M) \otimes \operatorname{CoT}_H(M).
$$
Using the definitions of $R^\prime$ and the multiplication on $\operatorname{CoT}_H(M)$, and the convolution invertibility of $R$, one can show that $R^\prime$ satisfies \eqref{eq:CQT1} for all elements of $\operatorname{CoT}_H(M) \otimes \operatorname{CoT}_H(M)$.
\end{proof}

Let $H$ be a graded coquasitriangular coquasi-Hopf algebra. A coquasitriangular structure $R$ on $H$ is said to be \textit{concentrated in degree $0$} if $R(x,y)=0$ for all homogeneous elements $x,y$ unless $\deg(x)=\deg(y)=0$. Clearly, the coquasitriangular structure on $\operatorname{CoT}_H(M)$ constructed in the proof of Proposition \ref{prop:CotHMCQT} is concentrated in degree $0$.

\begin{corollary}\label{coro:CQT}
Let $(H, R)$ be a cotriangular coquasi-Hopf algebra over $\k$, and let $M$ be an $H$-Majid bimodule. Then the coalgebra $\operatorname{CoT}_H(M)$ admits a cotriangular coquasi-Hopf algebra structure that is concentrated in degree $0$.
\end{corollary}
\begin{proof}
According to \cite[Proposition 2.16]{Yu26}, we have $$M\cong M^{co H}\otimes H.$$ Note that $M^{co H}$ is a left $H$-comodule. It then follows from Proposition \ref{prop:CotHMCQT} that the coalgebra $\operatorname{CoT}_H(M)$ admits a graded cotriangular coquasi-Hopf algebra structure that is concentrated in degree $0$.
\end{proof}

Let $(M, \rho_M^R)$ (respectively, $(N, \rho_N^L)$) be a finite-dimensional right (respectively, left) comodule over a coalgebra $H'$. Recall that the \textit{coefficient coalgebra} $\operatorname{cf}^R(M)$ (respectively, $\operatorname{cf}^L(N)$) of $M$ (respectively, $N$) is the smallest subcoalgebra of $H^\prime$ satisfying
\[
\rho_M^R(M) \subseteq M \otimes \operatorname{cf}^R(M)
\;\; \text{(respectively, } \rho_N^L(N) \subseteq \operatorname{cf}^L(N) \otimes N\text{)}.
\]
One can show that the following lemma holds.
\begin{lemma}\label{lem:RmidH}
Let $H$ be a coalgebra over $\k$, and let $M$ be an $H$-bicomodule. Suppose that the coalgebra $\operatorname{CoT}_H(M)$ admits the structure of a graded coquasitriangular coquasi-Hopf algebra whose coquasitriangular structure $R$ is concentrated in degree $0$. Then $H$ is a coquasitriangular coquasi-Hopf algebra. Moreover, for any elements $a, b \in \operatorname{cf}^L(M^{co H})$ or $a,b\in \operatorname{cf}^R({}^{co H}M)$, we have
$R(a_1,b_1)R(b_2, a_2)=\varepsilon(a)\varepsilon(b)$.
\end{lemma}
\begin{proof}
Since $\operatorname{CoT}_H(M)$ is a graded coquasi-Hopf algebra structure, it follows that $H$ is a coquasi-Hopf algebra.
A direct verification shows that $R\mid_{H\otimes H}$ is a coquasitriangular structure on $H$. According to \cite[Proposition 11.15(iii)]{BCPV19}, the coquasi-antipode of $H$ is bijective. By Lemma \ref{lem:gradedMajid} and \cite[Proposition 2.17]{Yu26}, we know that $M$ admits an $H$-Majid bimodule structure and $$H\otimes M^{co H}\cong M.$$ This implies that the set $$\{h\cdot m\mid h \text{ is a basis element of } H, m \text{ is a basis element of } M^{co H}\}$$ is linearly independent.
For any $m, n \in M^{\mathrm{co}H}$, by the fundamental theorem of comodules (\cite[Theorem 2.1.7]{DNR01}), $m$ and $n$ lie in finite-dimensional subcomodules $U$ and $V$ of $M^{\mathrm{co}H}$, respectively. Let $\{u_{1}, u_2, \dots, u_p\}$ and $\{v_{1}, v_2, \dots, v_q\}$ be bases of $U$ and $V$, respectively, such that
$$
\rho^L_U(u_i)=\sum_{j=1}^p c_{ij}\otimes u_j ,\;\; \rho^L_V(v_k)=\sum_{l=1}^q d_{kl}\otimes v_l.
$$
It then follows that
$$\Delta(c_{ij})=\sum\limits_{t=1}^p c_{it}\otimes c_{tj},\;\;\varepsilon(c_{ij})=\delta_{i, j}\;\; \text{for all } 1\leq i,j \leq p,$$
and
$$\Delta(d_{kl})=\sum\limits_{t=1}^q d_{kt}\otimes d_{tl},\;\;\varepsilon(d_{kl})=\delta_{k, l}\;\; \text{for all } 1\leq k,l \leq q.$$
In view of (\ref{eq:CQT1}) and the definition of the multiplication on $\operatorname{CoT}_H(M)$, we get
\begin{eqnarray*}
u_i\cdot v_k&=& \sum_{j=1}^p c_{ij} \cdot v_k \otimes u_j+\sum_{l=1}^q u_i\cdot  d_{kl} \otimes v_l\\
&=&\sum_{j=1}^p c_{ij} \cdot v_k \otimes u_j+ \sum_{r=1}^p\sum_{l,t=1}^q R(d_{kt}, c_{ir}) d_{tl}\cdot u_r \otimes v_l\\
&=&\sum_{j=1}^p  \sum_{l=1}^q R(d_{kl}, c_{ij}) v_l \cdot u_j\\
&=&\sum_{j=1}^p  \sum_{l=1}^qR(d_{kl}, c_{ij})(\sum_{r=1}^q d_{lr} \cdot u_j \otimes v_r +\sum_{s=1}^p v_l \cdot c_{js} \otimes u_s)\\
&=&\sum_{j=1}^p  \sum_{l,r=1}^qR(d_{kl}, c_{ij}) d_{lr} \cdot u_j \otimes v_r + \sum_{j,e,s=1}^p  \sum_{l,f=1}^q R(d_{kl}, c_{ij})R(c_{je}, d_{lf}) c_{es}\cdot v_f \otimes u_s.
\end{eqnarray*}
It follows that
$$
\sum_{j=1}^p c_{ij} \cdot v_k \otimes u_j=\sum_{j,e,s=1}^p  \sum_{l,f=1}^q R(d_{kl}, c_{ij})R(c_{je}, d_{lf}) c_{es}\cdot v_f \otimes u_s.
$$
Thus we have
$$
\sum_{j=1}^p  \sum_{l=1}^q R(d_{kl}, c_{ij})R(c_{je}, d_{lf})=\delta_{k,f} \delta_{i,e}.
$$
This means that
$$R(a_1,b_1)R(b_2, a_2)=\varepsilon(a)\varepsilon(b)$$ holds for all $a, b\in \operatorname{cf}^L(M^{co H})$. An analogous argument applies to $a,b\in \operatorname{cf}^R({}^{co H}M)$.
\end{proof}

\begin{proposition}\label{prop:CQTiff}
Let $H$ be a coquasitriangular coquasi-Hopf algebra over $\k$, and let $M$ be an $H$-Majid bimodule such that $H$ is generated by $\operatorname{cf}^L(M^{co H})$ or $\operatorname{cf}^R({}^{co H}M)$. Then the coalgebra $\operatorname{CoT}_H(M)$ admits a coquasitriangular coquasi-Hopf algebra structure concentrated in degree $0$ if and only if $H$ is cotriangular.
\end{proposition}

\begin{proof}
According to Corollary \ref{coro:CQT}, if $R(a_1,b_1)R(b_2, a_2)=\varepsilon(a)\varepsilon(b)$ for all $a,b\in H$, then we can construct a cotriangular coquasi-Hopf algebra structure on $\operatorname{CoT}_H(M)$ that is concentrated in degree $0$. Conversely, suppose that $\operatorname{CoT}_H(M)$ admits a coquasitriangular structure that is concentrated in degree $0$. Lemma \ref{lem:RmidH} implies that for any elements $a, b \in \operatorname{cf}^L(M^{co H})$ or $a,b\in \operatorname{cf}^R({}^{co H}M)$, we have
$R(a_1,b_1)R(b_2, a_2)=\varepsilon(a)\varepsilon(b)$. Since $H$ is generated by $\operatorname{cf}^L(M^{co H})$ or $\operatorname{cf}^R({}^{co H}M)$, it follows from (\ref{eq:CQT2}) and (\ref{eq:CQT3}) that $$R(a_1, b_1) R(b_2, a_2) = \varepsilon(a)\varepsilon(b)$$
for all $a, b \in H$.
\end{proof}
Note that the coquasitriangular structure in Proposition \ref{prop:CQTiff} is actually cotriangular. Moreover, we endow $M$ with a new Majid bimodule structure, thereby inducing a new coquasi-Hopf algebra structure.
It should also be emphasized that, since every Hopf algebra is a coquasi-Hopf algebra with trivial reassociator, the preceding conclusions immediately yield the corresponding results in the Hopf algebraic setting.
\subsection{Coquasitriangular and coribbon structures on coquasi-Hopf algebras}
We point out that, as the antipode of a coquasitriangular coquasi-Hopf algebra $H$ is bijective (\cite[Proposition 11.15(iii)]{BCPV19}), Remark \ref{rm:S(H0)<H0} implies $S(H_0)\subseteq H_0$. Consequently, by Lemma \ref{lem:gr}, $\operatorname{gr}H$ can be defined naturally. In Propositions \ref{prop:CotHMCQT} and \ref{prop:CQTiff}, we considered when $\operatorname{CoT}_{H_{[0]}}(H_{[1]}/H_{[0]})$ admits a coquasitriangular structure. A natural question then arises: does Theorem \ref{thm:embedding} have a coquasitriangular counterpart?

In preparation for answering this question, we first prove the following lemma.
\begin{lemma}\label{lem:grHCQT}
Let $(H,R)$ be a coquasitriangular coquasi-Hopf algebra over $\k$, and let $\operatorname{gr} H$ denote the graded coquasi-Hopf algebra constructed in Lemma \ref{lem:gr}. Define a linear map $\operatorname{gr} R:\operatorname{gr} H \otimes \operatorname{gr} H \rightarrow \k$ on homogeneous elements $a,b\in \operatorname{gr} H$ by
$$
\operatorname{gr} R(a,b):=
\begin{cases}
R(a,b), & \text{if } a,b\in H_{[0]}, \\
0, & \text{otherwise}.
\end{cases}
$$ Then $(\operatorname{gr} H, \operatorname{gr} R)$ remains coquasitriangular.
\end{lemma}
\begin{proof}
The proof is analogous to that of \cite[Lemma 5.25]{Yu26} in the coquasi-Hopf algebra case with the dual Chevalley property.
\end{proof}
We are now in a position to obtain the coquasitriangular version of Theorem \ref{thm:embedding}.
\begin{proposition}\label{prop:CQTembedding}
Let $(H,R)$ be a coquasitriangular coquasi-Hopf algebra over $\k$, and let $\operatorname{gr} H$ denote the graded coquasi-Hopf algebra constructed in Lemma \ref{lem:gr}. Then $\operatorname{CoT}_{H_{[0]}}(H_{[1]}/H_{[0]})$ admits a graded coquasitriangular coquasi-Hopf algebra structure, and there is a graded coquasitriangular coquasi-Hopf algebra embedding $\psi:\operatorname{gr} H\hookrightarrow \operatorname{CoT}_{H_{[0]}}(H_{[1]}/H_{[0]})$ such that $\psi\mid_{H_{[0]}\oplus H_{[1]}/H_{[0]}}=\id$.
\end{proposition}
\begin{proof}
By Lemma \ref{lem:grHCQT}, $(\operatorname{gr} H, \operatorname{gr} R)$ is coquasitriangular, where $\operatorname{gr} R$ is constructed in Lemma \ref{lem:grHCQT}. Theorem \ref{thm:embedding} implies that $$\psi:\operatorname{gr} H\hookrightarrow \operatorname{CoT}_{H_{[0]}}(H_{[1]}/H_{[0]})$$ is a graded coquasi-Hopf algebra embedding such that $\psi\mid_{H_{[0]}\oplus H_{[1]}/H_{[0]}}=\id$. It follows that $\operatorname{gr} R\mid_{H_{[0]} \otimes H_{[1]}/H_{[0]} \oplus H_{[1]}/H_{[0]} \otimes H_{[0]}} $ satisfies (\ref{eq:CQT1}).
Note that $R_{[0]}:=R\mid_{H_{[0]}\otimes H_{[0]}}$ is a coquasitriangular structure on $H_{[0]}$, making $(H_{[0]}, R_{[0]})$ a coquasitriangular coquasi-Hopf algebra. We extend $R_{[0]}$ trivially to be a function $R^\prime$ on $ \operatorname{CoT}_{H_{[0]}}(H_{[1]}/H_{[0]}) \otimes  \operatorname{CoT}_{H_{[0]}}(H_{[1]}/H_{[0]})$ defined by $$R^\prime(a,b):=\left\{
\begin{aligned}
R(a,b), &\text{ if} ~~~ a,b\in H_{[0]}; \\
0,\;\;\;\;&\text{ otherwise}.
\end{aligned}
\right.$$
We now claim that $(\operatorname{CoT}_{H_{[0]}}(H_{[1]}/H_{[0]}), R^\prime)$ is coquasitriangular. Indeed, the definition of $R^\prime$ immediately implies both its convolution invertibility and the identities \eqref{eq:CQT2} and \eqref{eq:CQT3}. Next we establish \eqref{eq:CQT1}. Since $\psi\mid_{H_{[0]}\oplus H_{[1]}/H_{[0]}}=\id$, it follows that $R^\prime \mid_{H_{[0]} \otimes H_{[1]}/H_{[0]} \oplus H_{[1]}/H_{[0]} \otimes H_{[0]}} $ satisfies (\ref{eq:CQT1}).
Note that the $H_{[0]}$-Majid bimodule structure on $H_{[1]}/H_{[0]}$ is induced by the multiplication of $\operatorname{gr} H$. By the definitions of $R^\prime$ and the multiplication on $\operatorname{CoT}_{H_{[0]}}(H_{[1]}/H_{[0]})$, and the convolution invertibility of $R$, we can show that (\ref{eq:CQT1}) holds for all elements $\sum x^1\otimes x^2\otimes \cdots \otimes x^n \in (H_{[1]}/H_{[0]})^{\Box n}$ and $\sum y^1\otimes y^2\otimes \cdots \otimes y^m \in (H_{[1]}/H_{[0]})^{\Box m}$.
\end{proof}

\begin{remark}\rm
It should be pointed out that in Corollary \ref{coro:CQT}, we considered the construction of a coquasitriangular coquasi-Hopf algebra structure on $\operatorname{CoT}_{H}(M)$ from a given cotriangular Hopf algebra $H$ and an $H$-Majid bimodule $M$. In particular, Proposition \ref{prop:CQTembedding} shows that $H_{[0]}$ forms a coquasi-Hopf algebra and $H_{[1]}/H_{[0]}$ forms an $H_{[0]}$-Majid bimodule, and it also provides a coquasitriangular structure on $\operatorname{CoT}_{H_{[0]}}(H_{[1]}/H_{[0]})$. It should be noted, however, that the multiplication is defined differently in these two constructions. Moreover, we remark that the existence of a coquasitriangular structure $R^\prime$ on $\operatorname{CoT}_H(M)$ concentrated in degree $0$, with $R^\prime\mid_{H\otimes H}$ being the coquasitriangular structure on $H$, can be proved by checking only that $R^\prime\mid_{H\otimes M\oplus M\otimes H}$ satisfies \eqref{eq:CQT1}.
\end{remark}

\begin{corollary}
Let $(H,R)$ be a coquasitriangular coquasi-Hopf algebra over $\k$, and let $\operatorname{gr} H$ denote the graded coquasi-Hopf algebra constructed in Lemma \ref{lem:gr}. Then for any elements $a, b \in \operatorname{cf}^L((H_{[1]}/H_{[0]})^{co H_{[0]}})$ or $a,b\in \operatorname{cf}^R({}^{co H_{[0]}}(H_{[1]}/H_{[0]}))$, we have
$R(a_1,b_1)R(b_2, a_2)=\varepsilon(a)\varepsilon(b)$. In particular, $H_{[0]}$ is cotriangular whenever it is generated by either $\operatorname{cf}^L((H_{[1]}/H_{[0]})^{co H_{[0]}})$ or $\operatorname{cf}^R({}^{co H_{[0]}}(H_{[1]}/H_{[0]}))$.
\end{corollary}
\begin{proof}
    This follows directly from Lemma \ref{lem:RmidH}, Propositions \ref{prop:CQTiff} and \ref{prop:CQTembedding}.
\end{proof}

\begin{example}\rm
We claim that the coquasi-Hopf algebra $M(A,X,z,g,y)$ constructed in Example \ref{ex:M(A,X,z,g,y)} does not admit a coquasitriangular structure. Indeed, if $M(A,X,z,g,y)$ admits a coquasitriangular structure $R$, then we have
$$
R(A^2, X)^2 = R(A^2, X^2) = 0,
$$
and hence
$$
1 = R(A^4, A) = R(A^2, A)^2 = R(A^2, A^2).
$$
Applying \eqref{eq:CQT1} to $A^2$ and $X$, we obtain
$$
R(A^2, X) A^2 +R(A^2,A^2)A^2X= XA^2 R(A^2, 1)+ R(A^2,X).
$$
This means that $A^2X=XA^2$, which is a contradiction.
\end{example}

Next, we give an example of a coquasitriangular coquasi-Hopf algebra.
\begin{example}\rm\label{eq:M(a,b,c,d,g)}
Let $\k$ be an algebraically closed field of characteristic $0$, and let $q$ be a primitive $3$rd root of unity. We consider a coquasi-Hopf algebra structure on $$\mathcal{O}_q(SL_2)/(a^3-1,d^3-1,b^3,c^3) \otimes (\k \mathbb{Z}_2, \omega),$$ where
$
\mathcal{O}_q(SL_2)/(a^3-1,d^3-1,b^3,c^3)
$
is the finite-dimensional quantum coordinate algebra at a primitive $3$rd root of unity (see \cite{AM19}). Specifically, let $M(a,b,c,d,g)$ be the $54$-dimensional
coquasi-Hopf algebra over $\k$ generated by $a,b,c,d,g$ with relations
\begin{eqnarray*}
&ba=qab,\;\; ca=qac,\;\; db=qbd,\;\; dc=qcd,\;\; bc=cb,\\
&ad-q^{2}bc=1,\;\; da-ad=(q-q^{-1})bc,\\
&a^3=d^3=g^2=1,\;\; b^3=c^3=0,\\
&ga=ag,\;\; gb=bg,\;\; gc=cg,\;\; gd=dg.
\end{eqnarray*}
Its coalgebra structure, reassociator, and coquasi-antipode are given by
\begin{eqnarray*}
&&\Delta(a)=a\otimes a+b\otimes c,\;\;\varepsilon(a)=1,\;\; S(a)=d,\\
&&\Delta(b)=a\otimes b+b\otimes d,\;\;\varepsilon(b)=0,\;\; S(b)=-qb, \\
&&\Delta(c)=c\otimes a+d\otimes c,\;\;\varepsilon(c)=0, \;\; S(c)=-q^2c,\\
&&\Delta(d)=c\otimes b+d\otimes d,\;\;\varepsilon(d)=1,\;\; S(d)=a,\\
&&\Delta(g)=g\otimes g,\;\;\varepsilon(g)=1,\;\; S(g)=g,\\
&&\alpha(a^{r_a}b^{r_b}c^{r_c}d^{r_d}g^{r_g})=\delta_{r_b,0}\delta_{r_c,0},\;\; \beta(a^{r_a}b^{r_b}c^{r_c}d^{r_d}g^{r_g})=(-1)^{r_g}\delta_{r_b,0}\delta_{r_c,0},\\
&&\Phi(a^{r_a}b^{r_b}c^{r_c}d^{r_d}g^{r_g}, a^{s_a}b^{s_b}c^{s_c}d^{s_d}g^{s_g}, a^{t_a}b^{t_b}c^{t_c}d^{t_d}g^{t_g})=(-1)^{r_gs_gt_g}\delta_{r_b,0}\delta_{r_c,0}\delta_{s_b,0}\delta_{s_c,0}\delta_{t_b,0}\delta_{t_c,0}.
\end{eqnarray*}
Since $d=a^2(1+q^2bc)$, a basis of $M(a,b,c,d,g)$ is given by
$$
\{a^{r_a}b^{r_b}c^{r_c}g^{r_g}\mid 0\le r_a,r_b,r_c\le 2,\ 0\le r_g\le 1\}.
$$
We have $$M(a,b,c,d,g)_0=\k 1 \oplus \k g\oplus  C_1 \oplus D_1\oplus C_g\oplus D_g,$$
where the comatrix coalgebras associated with $C_1, D_1, C_g$ and $D_g$ are represented by the matrices
\begin{eqnarray*}
\C_1=\begin{pmatrix}
a & b \\
c & d
\end{pmatrix},\;\;
\D_1=\begin{pmatrix}
a^2 & ab & b^2\\
(1+q^2)ac & ad+qbc & (1+q^2)bd\\
c^2 & cd & d^2
\end{pmatrix},
\end{eqnarray*}
\begin{eqnarray*}
\C_g=\begin{pmatrix}
ag & bg \\
cg & dg
\end{pmatrix},\;\;
\D_g=\begin{pmatrix}
a^2g & abg & b^2g\\
(1+q^2)acg & (ad+qbc)g & (1+q^2)bdg\\
c^2g & cdg & d^2g
\end{pmatrix},
\end{eqnarray*} and $C_1, D_1, C_g$ and $D_g$ are precisely the linear spans of the entries of the respective matrices. Although $M(a,b,c,d,g)$ does not possess the dual Chevalley property, it is nevertheless generated as an algebra by $M(a,b,c,d,g)_0$. Moreover, $M(a,b,c,d,g)$ is coquasitriangular, and a coquasitriangular structure $R$ on $M(a,b,c,d,g)$ can be constructed by
$$
R(a^{r_a}b^{r_b}c^{r_c}g^{r_g}, a^{s_a}b^{s_b}c^{s_c}g^{s_g})
=\sqrt{-1}^{r_g s_g} r(a^{r_a}b^{r_b}c^{r_c}, a^{s_a}b^{s_b}c^{s_c}).
$$
Here $r$ is defined on the generators $a,b,c$ by the following table, motivated by (4.3) in the proof of \cite[Theorem 4.1(2)]{AM19}:
$$
\begin{array}{|c|c|c|c|c|}
\hline
r(x,y) & a & b & c  \\ \hline
a & q^2 & 0 & 0  \\\hline
b & 0 & 0 & q^2-1  \\\hline
c & 0 & 0 & 0  \\
\hline
\end{array}
$$
Then by
\begin{eqnarray*}
&r(xy,z)=\sum r(x,z_{1})r(y,z_{2}),\\
&r(x,yz)=\sum r(x_{1},z)r(x_{2},y),\\
&r(1,x)=r(x,1)=\varepsilon(x),
\end{eqnarray*}
we obtain all values of $r$ on $\mathcal{O}_q(SL_2)/(a^3-1,d^3-1,b^3,c^3)$. We may consider $k\langle x\rangle \# M(a,b,c,d,g)$ to obtain the coquasi-Hopf algebra $M(a,b,c,d,g,x)$. Specifically, we adjoin to $M(a,b,c,d,g)$ a generator $x$ subject to the relations
$$
ax=xa,\;\; bx=xb,\;\; cx=xc,\;\; dx=xd,\;\; gx=xg,
$$
where $x$ is a primitive element. We can show that $M(a,b,c,d,g,x)$ is a coquasi-Hopf algebra whose reassociator is concentrated in degree $0$, and $$M(a,b,c,d,g,x)_{[0]}=M(a,b,c,d,g).$$ Note that $M(a,b,c,d,g,x)$ admits a coquasitriangular structure concentrated in degree $0$.
\end{example}

The definition of coribbon Hopf algebras can be found in \cite[Definition 3]{Wa03}. We may therefore naturally give the following definition of coribbon coquasi-Hopf algebras.
\begin{definition}
A coribbon coquasi-Hopf algebra is a coquasitriangular
coquasi-Hopf algebra $(H,R)$ equipped with a convolution-invertible
functional $\nu:H\to k$ such that
\begin{equation}
\nu(1)=1,\;\; \nu\circ S=\nu, \label{eq:coribbon1}
\end{equation}
\begin{equation}
\nu(h_{(1)})h_{(2)}=h_{(1)}\nu(h_{(2)}),\label{eq:coribbon2}
\end{equation}
\begin{equation}
\nu(ab)=R^{-1}(a_{(1)},b_{(1)})
\nu(a_{(2)})\nu(b_{(2)})
R^{-1}(b_{(3)},a_{(3)}),\label{eq:coribbon3}
\end{equation}
for all $a,b,h\in H$.
\end{definition}
Clearly, a cotriangular coquasi-Hopf algebra is necessarily coribbon, with $\nu$ taken to be $\varepsilon$. According to Corollary \ref{coro:CQT}, $\operatorname{CoT}_H(M)$ is coribbon in the case that $H$ is cotriangular and $M$ is an $H$-Majid bimodule.

We next prove that $\operatorname{gr}H$ is coribbon whenever $H$ is coribbon.
\begin{lemma}\label{lem:grcoribbon}
Let $(H,R,\nu)$ be a coribbon coquasi-Hopf algebra over $\k$, and let $\operatorname{gr} H$ denote the graded coquasi-Hopf algebra constructed in Lemma \ref{lem:gr}. Let $\operatorname{gr}\nu$ be the linear map $\operatorname{gr}H\to \k$ that equals $\nu$ on $H_{[0]}$ and vanishes otherwise.
Then $(\operatorname{gr} H, \operatorname{gr} R, \operatorname{gr} \nu)$ is a coribbon coquasi-Hopf algebra.
\end{lemma}
\begin{proof}
By Lemma \ref{lem:grHCQT}, $(\operatorname{gr} H, \operatorname{gr} R)$ is coquasitriangular. Next, we show that $\operatorname{gr}\nu$ defines a coribbon structure. It suffices to verify (\ref{eq:coribbon2}) for $\operatorname{gr}\nu$, since (\ref{eq:coribbon1}), (\ref{eq:coribbon3}), and the convolution invertibility of $\operatorname{gr}\nu$ are obvious. For any $\bar{a} \in H_{[n]}/H_{[n-1]}$, there exist $b \in H_{[n-1]} $ such that $a+b\in H_{[n]}$. Then by the assumption that $\nu$ is a coribbon structure on $H$, we have $$\nu(a_1) a_2+\nu(b_1)b_2= a_1\nu(a_2)+b_1 \nu(b_2),$$
which implies that
$$
(\nu(a_1) a_2+\nu(b_1)b_2)/H_{[n-1]}=( a_1\nu(a_2)+b_1 \nu(b_2)) /H_{[n-1]}.
$$
Thus (\ref{eq:coribbon2}) follows.
\end{proof}
We can also obtain the coribbon version of Theorem \ref{thm:embedding}.
\begin{proposition}\label{prop:coribbonembedding}
Let $(H,R,\nu)$ be a coribbon coquasi-Hopf algebra over $\k$, and let $\operatorname{gr} H$ denote the graded coquasi-Hopf algebra constructed in Lemma \ref{lem:gr}. Then $\operatorname{CoT}_{H_{[0]}}(H_{[1]}/H_{[0]})$ admits a graded coribbon coquasi-Hopf algebra structure, and there is a graded coribbon coquasi-Hopf algebra embedding $\psi:\operatorname{gr} H\hookrightarrow \operatorname{CoT}_{H_{[0]}}(H_{[1]}/H_{[0]})$ such that $\psi\mid_{H_{[0]}\oplus H_{[1]}/H_{[0]}}=\id$.
\end{proposition}
\begin{proof}
By Proposition \ref{prop:CQTembedding}, there is a graded coquasitriangular coquasi-Hopf algebra embedding $\psi:\operatorname{gr} H\hookrightarrow \operatorname{CoT}_{H_{[0]}}(H_{[1]}/H_{[0]})$ such that $\psi\mid_{H_{[0]}\oplus H_{[1]}/H_{[0]}}=\id$. Define $$\nu^\prime(a):=\left\{
\begin{aligned}
\nu(a), &\text{ if} ~~~ a\in H_{[0]}; \\
0,\;\;\;\;&\text{ otherwise}.
\end{aligned}
\right.$$ Next, we show that $\nu^\prime$ is a coribbon structure on $\operatorname{CoT}_{H_{[0]}}(H_{[1]}/H_{[0]})$.
Observe that $H_{[1]}/H_{[0]}$ carries an $H_{[0]}$-bicomodule structure induced by the comultiplication. For every $x\in H_{[1]}/H_{[0]}$, we denote the bicomodule structure by
$$
\rho_{ H_{[1]}/H_{[0]}}^L(x)=x_{-1}\otimes x_0,\;\; \rho_{ H_{[1]}/H_{[0]}}^R(x)=x_0\otimes x_1.
$$
Lemma \ref{lem:grcoribbon} implies that $\operatorname{gr}H$ is coribbon. Consequently, for every $x\in H_{[1]}/H_{[0]}$, we have
$$
\nu(x_{-1}) x_0=x_0\nu (x_1).
$$
For $\sum x^1\otimes x^2\otimes \cdots \otimes x^n\in (H_{[1]}/H_{[0]})^{\Box n}$, we have
\begin{eqnarray*}
  \sum \nu(x^1_{-1}) x^1_0  \otimes x^2\otimes \cdots \otimes x^n&=&\sum x^1_{0}\nu( x^1_1)  \otimes x^2\otimes \cdots \otimes x^n\\
  &=& \sum x^1 \nu (x_{-1}^2) \otimes x^2_0\otimes \cdots \otimes x^n\\
  &=& \cdots\\
  &=& \sum x^1\otimes x^2\otimes \cdots \otimes x^{n-1}\otimes \nu (x_{-1}^{n}) x^{n}_0\\
  &=& \sum x^1\otimes x^2\otimes \cdots \otimes x^{n-1}\otimes x^{n}_0\nu(x^{n}_1).
\end{eqnarray*}
This means that $\nu^\prime$ satisfies (\ref{eq:coribbon2}). Moreover, the definition of $\nu^\prime$ immediately implies its convolution invertibility as well as the identities \eqref{eq:coribbon1} and \eqref{eq:coribbon3}. Hence $\nu^\prime$ endows $\operatorname{CoT}_{H_{[0]}}(H_{[1]}/H_{[0]})$ with a coribbon structure, making $\psi$ a morphism of coribbon coquasi-Hopf algebras.
\end{proof}

In particular, when the reassociator is trivial, the above results reduce to their Hopf algebra counterparts.
\begin{corollary}
Let $H$ be a Hopf algebra over $\k$ with antipode $S$ such that $S(H_0)\subseteq H_0$. Then $\operatorname{CoT}_{H_{[0]}}(H_{[1]}/H_{[0]})$ admits a graded Hopf algebra structure, and there is a graded Hopf algebra embedding $\psi:\operatorname{gr} H\hookrightarrow \operatorname{CoT}_{H_{[0]}}(H_{[1]}/H_{[0]})$ such that $\psi|_{H_{[0]}\oplus H_{[1]}/H_{[0]}}=\operatorname{id}$. In particular, if $H$ is coquasitriangular (respectively, coribbon), then this embedding is a morphism of coquasitriangular (respectively, coribbon) Hopf algebras.
\end{corollary}
Let $\k$ be an algebraically closed field.
According to \cite[Proposition 2.6]{EO04}, a finite tensor category $\mathscr{C}$ is integral, i.e., the Frobenius-Perron dimensions of all its objects are integers, if and only if $\mathscr{C}$ is tensor equivalent to the comodule category $\mathscr{M}^H$ of a finite-dimensional coquasi-Hopf algebra $H$. Moreover, by \cite[Proposition 10.1]{BCPV19} (respectively, \cite[Proposition 13.20(ii)]{BCPV19}), a finite integral braided (respectively, ribbon) tensor category is tensor equivalent to the comodule category of a finite-dimensional coquasitriangular (respectively, coribbon) coquasi-Hopf algebra. For details on ribbon tensor categories, see \cite{EGNO15}.
Consequently, we obtain a characterization of finite integral tensor categories.
\begin{corollary}
Let $\mathscr{C}$ be a finite integral tensor category over an algebraically closed field $\k$. Then $\mathscr{C}$ is tensor equivalent to $\mathscr{M}^H$ for some finite-dimensional coquasi-Hopf algebra $H$. Moreover, there is an embedding $\mathscr{M}^{\operatorname{gr} H}\hookrightarrow \mathscr{M}^{\operatorname{CoT}_{H_{[0]}}(H_{[1]}/H_{[0]})}$ of tensor categories, and $\mathscr{M}^{\operatorname{gr} H}$ contains all objects of
$\mathscr{M}^{\operatorname{CoT}_{H_{[0]}}(H_{[1]}/H_{[0]})}$ of Loewy length $\leq 2$, where $\operatorname{gr} H$ is the graded coquasi-Hopf algebra associated with the standard filtration. In particular, if $\mathscr{C}$ is braided (respectively, ribbon), then $H$ is coquasitriangular (respectively, coribbon) and the embedding above is a braided (respectively, ribbon) tensor functor.
\end{corollary}
\begin{proof}
By \cite[Corollary 4.4]{BC03}, a coquasi-antipode of a finite-dimensional coquasi-Hopf algebra is bijective. It follows from Lemmas \ref{lem:Majidcoradical} and \ref{lem:gr} that $H_{[0]}$ is a coquasi-Hopf subalgebra of $H$ and $\operatorname{gr} H$ is well-established. Then by Lemma \ref{lem:antipodebijective}, the coquasi-antipode of $\operatorname{CoT}_{H_{[0]}}(H_{[1]}/H_{[0]})$ is bijective and thus $\mathscr{M}^{\operatorname{CoT}_{H_{[0]}}(H_{[1]}/H_{[0]})}$ forms a tensor category. The remainder is an immediate consequence of \cite[Proposition 2.6]{EO04}, \cite[Propositions 10.1 and 13.20(ii)]{BCPV19}, Theorem \ref{thm:embedding}, Propositions \ref{prop:CQTembedding} and \ref{prop:coribbonembedding}.
\end{proof}

\begin{example}\rm
According to \cite[Section 3.4]{Maj95}, the quantum group $u_q(\mathfrak{sl}_2)$ is a ribbon Hopf algebra when $q$ is a primitive third root of unity. By \cite[Proposition A.6]{Cli19},
$$
u_q(\mathfrak{sl}_2)^*\cong \mathcal{O}_q(SL_2)/(a^3-1,d^3-1,b^3,c^3),
$$
so the latter is coribbon. Let $\nu_0$ be the coribbon structure of
$\mathcal{O}_q(SL_2)/(a^3-1,d^3-1,b^3,c^3)$
defined on the generators $a,b,c$ by
$$
\nu_0(a)=1,\;\; \nu_0(b)=0,\;\; \nu_0(c)=0,
$$
and extended by (\ref{eq:coribbon3}).
Then define a linear map $\nu:M(a,b,c,d,g)\to \k$ by
$$
\nu(a^{r_a}b^{r_b}c^{r_c}g^{r_g})
=\nu_0(a^{r_a}b^{r_b}c^{r_c})(\sqrt{-1})^{r_g}.
$$
One can directly verify that $\nu$ defines a coribbon structure on $M(a,b,c,d,g)$ from Example \ref{eq:M(a,b,c,d,g)}.
\end{example}

\section{Quasi-coquasi-Hopf pairings}\label{Section5}
\subsection{Preliminaries on tensor quasi-Hopf algebras}
In this subsection, we recall the basics of quasi-Hopf algebras.

Recall that a \textit{quasi-bialgebra} $(A, m, \mu, \Delta, \varepsilon, \phi)$ (\cite[Definition 3.4]{BCPV19}) is a $\k$-algebra $(A, m, \mu)$ with algebra morphisms $\Delta: A \to A \otimes A$ and $\varepsilon: A \to \k$, and an invertible element $\phi \in A \otimes A \otimes A$ (also called associator), such that for $a\in A,$ the following relations hold
\begin{equation*}
(\id \otimes \Delta)\Delta(a) = \phi\big((\Delta \otimes \id)\Delta(a)\big)\phi^{-1},
\end{equation*}
\begin{equation*}
(\id \otimes \id \otimes \Delta)(\phi)(\Delta \otimes \id \otimes \id)(\phi) = (1 \otimes \phi)(\id \otimes \Delta \otimes \id)(\phi)(\phi \otimes 1),
\end{equation*}
\begin{equation*}
(\varepsilon \otimes \id)\Delta = \id = (\id \otimes \varepsilon)\Delta,
\end{equation*}
\begin{equation*}
(\id \otimes \varepsilon \otimes \id)(\phi) = 1 \otimes 1.
\end{equation*}
We denote $\phi =  X^1 \otimes X^2 \otimes X^3$ and $\phi^{-1} =  \overline{X}^1 \otimes \overline{X}^2 \otimes \overline{X}^3$. A quasi-bialgebra $A$ is called a \textit{quasi-Hopf algebra} (\cite[Definition 3.13]{BCPV19}) if there is a linear algebra antimorphism $S: A \to A$ and elements $\alpha, \beta \in A$ satisfying for all $a \in A$,
\begin{equation*}
 S(a_{1})\alpha a_{2} =\varepsilon(a) \alpha , \;\;  a_{1}\beta S(a_{2}) = \varepsilon(a) \beta ,
\end{equation*}
\begin{equation*}
 X^1 \beta S(X^2) \alpha X^3 = 1 =  S(\overline{X}^1) \alpha \overline{X}^2 \beta S(\overline{X}^3).
\end{equation*}

Notice that the axioms for a quasi-Hopf algebra immediately imply that $(\varepsilon\otimes \id \otimes \id)(\phi)=( \id \otimes \id \otimes\varepsilon)(\phi)=1\otimes 1,$ $\varepsilon \circ S=\varepsilon,$ and $\varepsilon(\alpha)\varepsilon(\beta)=1.$ For details, see (\cite[Chapter 3]{BCPV19}).
Moreover, we know that $S$ is a coalgebra antimorphism up to twist. To be more specific, let
\begin{equation}
F= (S\otimes S) (\Delta^{cop}(\overline{X}^1))\Gamma \Delta(\overline{X}^2\beta S(\overline{X}^3)), \label{eq:F}
\end{equation}
where \begin{equation}
\Gamma=S(X^2 \overline{X}^1_2) \alpha X^3 \overline{X}^2\otimes S(X^1 \overline{X}^1_1)\alpha \overline{X}^3. \label{eq:Gamma}
\end{equation}
Then by \cite[Theorem 3.21]{BCPV19}, we have $$F\Delta(S(a))F^{-1}=(S\otimes S)(\Delta^{cop}(a)). $$

Recall from \cite[Section 2]{AC92} that a \textit{quasitriangular} quasi-Hopf algebra is a quasi-Hopf algebra
$(A,m,\mu,\Delta,\varepsilon,\phi,S,\alpha,\beta)$ equipped with an invertible element $R\in A\otimes A$ satisfying
\begin{equation*}
\Delta^{cop}(a)R=R\Delta(a), \label{eq:QT1}
\end{equation*}
\begin{equation*}
(\Delta\otimes\operatorname{id})(R)=\phi_{[231]}R^{13}\phi_{[132]}^{-1}R^{23}\phi,\label{eq:QT2}
\end{equation*}
\begin{equation*}
(\operatorname{id}\otimes\Delta)(R)=\phi_{[312]}^{-1}R^{13}\phi_{[213]}R^{12}\phi^{-1}\label{eq:QT3}
\end{equation*}
for every $a\in A$. Here, for $T= T^1\otimes T^2\otimes T^3$, we write
$T_{[ijk]}= T^i\otimes T^j\otimes T^k$. For $R= r^1\otimes r^2$, set $R^{12}= r^1\otimes r^2 \otimes 1, R^{13}= r^1\otimes 1\otimes r^2, R^{23}=1\otimes  r^1\otimes r^2.$
A \textit{ribbon} quasi-Hopf algebra (\cite[Section 4]{AC92}) is a quasitriangular
quasi-Hopf algebra $(A,R)$ equipped with an invertible central
element $v\in A$ satisfying
\begin{equation*}
\Delta(v)=(v\otimes v)(R^{21}R)^{-1}, \;\;
S(v)=v,\label{eq:ribbon2}
\end{equation*}
where $R^{21}= r^2\otimes r^1.$ It follows that $\varepsilon(v)=1$ (see \cite[Remark 13.27]{BCPV19}).

It is well known that the dual $A^*$ of a finite-dimensional quasitriangular (respectively, ribbon) quasi-Hopf algebra $A$ is a coquasitriangular (respectively, coribbon) coquasi-Hopf algebra.

\begin{example}\rm
\begin{itemize}
\item[(1)] According to \cite[Theorem 10.20]{BCPV19}, the Drinfeld double of a finite-dimensional quasi-Hopf algebra is a quasitriangular quasi-Hopf algebra. Dually, we obtain a coquasitriangular coquasi-Hopf algebra. Let $G$ be a finite group and $\omega$ a $3$-cocycle on $G$ such that the characteristic of $\k$ does not divide the order of $G$. Dijkgraaf, Pasquier and Roche \cite{DPR90} constructed a quasi-Hopf algebra $D^\omega(G)$, known as the \textit{twisted Drinfeld double} of $G$. Note that the coquasi-Hopf algebra $(D^\omega(G))^*$ is both coribbon and cosemisimple \cite[Section 5]{AC92}.
\item[(2)]For a simple finite-dimensional Lie algebra $\mathfrak{g}$ over $\mathbb{C}$, Drinfeld \cite[Proposition 3.16]{Dri90} showed that a quasitriangular quasi-Hopf quantized enveloping algebra $U\mathfrak{g}[[h]]$ is indeed twist equivalent to the usual quantum group $U_h\mathfrak{g}$. Liu, Van Oystaeyen, and Zhang \cite{Liu14, LVZ17} constructed the quasi-Frobenius-Lusztig kernels $Qu_q(\mathfrak{g})$ associated with any simple finite-dimensional Lie algebra $\mathfrak{g}$. Inspired by the classical Frobenius-Lusztig kernel theory \cite{Lus90}, they expected that the quasi-Frobenius-Lusztig kernel associated with $\mathfrak{g}$ should be the Drinfeld double of the half small quasi-quantum group defined in \cite{EG05}. They also established that, in many cases, $Qu_q(\mathfrak{g})$ is not twist-equivalent to any Hopf algebra. By taking duals, we obtain examples of genuine non-cosemisimple coquasitriangular coquasi-Hopf algebras.
\item[(3)]In \cite[Subsection 3.1]{FGR22}, Farsad, Gainutdinov and Runkel constructed a class of finite-dimensional quasitriangular quasi-Hopf algebras $Q(N)$. It is worth mentioning that the quasi-Hopf algebras $Q(N)$ are ribbon. They also showed that ${}_{Q(N)}\mathscr{M}$ is ribbon equivalent to the symplectic fermion category $\mathcal{SF}(N)$ \cite[Theorem 3.6]{FGR22}, which was computed in \cite{Run14} from the conformal blocks of the corresponding logarithmic conformal field theory. According to Proposition \ref{prop:coribbonembedding}, $\mathscr{M}^{\operatorname{gr}(Q(N))^*}$ can be embedded, as a ribbon tensor category, into the comodule category of a cotensor coquasi-Hopf algebra.
\end{itemize}
\end{example}

Let $H$ be a Hopf algebra and $M$ a Hopf bimodule over $H$. When both $H$ and $M$ are finite-dimensional, $H^*$ inherits a Hopf algebra structure, and $M^*$ becomes a Hopf bimodule over $H^*$. Moreover, the graded dual of $\operatorname{CoT}_H(M)$ is isomorphic to the tensor Hopf algebra
$T_{H^*}(M^*).$
For details, see \cite{Nic78}.
Bulacu \cite{Bul20} extended Nichols' tensor Hopf algebra construction to the quasi-Hopf algebra setting. In Proposition \ref{prop:CotCQH}, we parallel this by extending Nichols' construction of Hopf algebras over $\operatorname{CoT}_H(M)$ to the coquasi-Hopf algebra case.
We next recall the tensor quasi-Hopf algebra structures constructed by Bulacu \cite{Bul20}. Before doing so, we first recall the definition of a two-sided two-cosided quasi-Hopf bimodule, which is dual to that of a Majid bimodule.

Let $A$ be a quasi-bialgebra, and let $M$ be an $A$-bimodule equipped with two $A$-bimodule maps $\rho^L_M : M \rightarrow A \otimes M$ and $\rho^R_M : M \rightarrow M \otimes A$. We call $(M, \rho^L_M, \rho^R_M)$ a \textit{two-sided two-cosided quasi-Hopf $A$-bimodule} (\cite[Definition 3.1]{Sch02}) if $M$ is an $A$-bicomodule in the monoidal category ${}_A\mathcal{M}_A$, that is, if the following conditions hold for all $m \in M$:
\begin{equation*}
(\varepsilon \otimes \id_M) \circ \rho^L_M = \id_M,\;\; (\id_M \otimes \varepsilon) \circ \rho^R_M = \id_M,
\end{equation*}
\begin{equation*}
(\id_A \otimes \rho^L_M)\rho^L_M(m) \cdot \phi = \phi \cdot (\Delta \otimes \id_M)\rho^L_M(m),
\end{equation*}
\begin{equation*}
\phi \cdot (\rho^R_M \otimes \id_A)\rho^R_M(m) = (\id_M \otimes \Delta)\rho^R_M(m) \cdot \phi,
\end{equation*}
\begin{equation*}
\phi \cdot (\rho^L_M \otimes\id_A)\rho^R_M(m) = (\id_A \otimes \rho^R_M)\rho^L_M(m)\cdot \phi.
\end{equation*}

We now take $A$ to be a quasi-Hopf algebra and $M$ to be a two-sided two-cosided quasi-Hopf $A$-bimodule. Denote $T_A(M)$ be the tensor algebra associated to the $A$-bimodule $M$, i.e.,
$$T_A(M) = A \oplus \bigoplus_{n\geq 1} M^{\otimes_A n},$$ where $M^{\otimes_A 1} := M$ and $M^{\otimes_A n} := M^{\otimes_A n-1} \otimes_A M$, for all $n \geq 2$. The multiplication $\circ$ on $T_A(M)$ is given by
\begin{eqnarray*}
    &&a\circ a^\prime =aa^\prime,\\
    &&a \circ(m^1 \otimes_A m^2 \otimes_A \cdots \otimes_A m^l)=(a\cdot m^1)\otimes_A m^2 \otimes_A \cdots \otimes_A m^l,\\
    &&(m^1 \otimes_A m^2 \otimes_A \cdots \otimes_A m^l) \circ a=m^1 \otimes_A m^2 \otimes_A \cdots \otimes_A (m^l\cdot a),\\
   && (m^1 \otimes_A m^2 \otimes_A \cdots \otimes_A m^l)\circ(m^{l+1} \otimes_A m^{l+2} \otimes_A \cdots \otimes_A m^{n})=m^1 \otimes_A m^2 \otimes_A \cdots \otimes_A m^n,
\end{eqnarray*}
for all $a, a^\prime \in A$, $l \geq 1$, $n \geq 2$ and $m^1, \dots, m^n \in M$. The unit of $T_A(M)$ is $1_A$. To avoid confusion, we will write $T_A(M)\underline{\otimes} T_A(M)$ for $T_A(M)\otimes T_A(M)$. The comultiplication $\tilde{\Delta}$ on $T_A(M)$ is defined on generators by $\tilde{\Delta}(a) = \Delta(a)$ for $a \in A$, and
$$
\tilde{\Delta}(m) = \rho^L_M(m) + \rho^R_M(m) = m_{-1}\otimes m_{0} + m_{0}\otimes m_{1} \in T_A(M)\underline{\otimes} T_A(M),
$$
for $m \in M$. By the universal property of $T_A(M)$ (see \cite[Proposition 1.4.1]{Nic78}), it extends to an algebra morphism $\tilde{\Delta}: T_A(M) \to T_A(M)\underline{\otimes} T_A(M)$. The counit $\tilde{\varepsilon}$ is defined on the generators by $\tilde{\varepsilon}(a)=\varepsilon(a)$ for $a\in A$ and $\tilde{\varepsilon}(m)=0$ for $m\in M$, and extends uniquely to an algebra morphism on $T_A(M)$. For the associator, we have $\tilde{\phi} = X^1\underline{\otimes} X^2\underline{\otimes} X^3$, where $\phi = X^1 \otimes X^2 \otimes X^3$ is the associator of $A$. Given an antipode $(S, \alpha, \beta)$ of $A$, one obtains an antipode $(\tilde{S}, \tilde{\alpha}, \tilde{\beta})$ for $T_A(M)$ by putting $\tilde{\alpha} = \alpha$, $\tilde{\beta} = \beta$, $\tilde{S}(a) = S(a)$ for all $a \in A$, and
$$
\tilde{S}(m) = - S(m_{0,-1} \overline{X}^1) \alpha \cdot m_{0,0} \cdot \overline{X}^2 \beta S(\overline{X}^3) S(m_{1}),
$$
for all $m\in M$, where $
(\rho^L_M \otimes\id_A)\rho^R_M(m)=m_{0,-1}\otimes m_{0,0}\otimes m_1
$. This assignment extends to an algebra antimorphism from $T_A(M)$ to itself.

\begin{lemma}\emph{(}\cite[Proposition 5.4]{Bul20}\emph{)}\label{lem:tensorquasi}
With the above notation, $(T_A(M), \circ, 1_A, \tilde{\Delta}, \tilde{\varepsilon}, \tilde{\phi}, \tilde{S}, \tilde{\alpha}, \tilde{\beta})$ is a graded quasi-Hopf algebra. Such a quasi-Hopf algebra is called a tensor quasi-Hopf algebra associated with $A$ and $M$.
\end{lemma}
In the next subsection, we will present the connection between the above tensor quasi-Hopf algebra and the cotensor coquasi-Hopf algebras established in Subsection \ref{subsection3.1}.

\subsection{Pairings between tensor quasi-Hopf and cotensor coquasi-Hopf algebras}
In \cite[p.110]{Kas95}, the notion of a Hopf pairing is introduced; a classical example is the Hopf pairing between the quantum group $U_q(\mathfrak{sl}_2)$ and $\mathcal{O}_q(SL_2)$ (see \cite[Theorem VII.4.4]{Kas95}). In \cite[Proposition 3.4]{Che07}, Chen established a graded Hopf pairing between certain tensor Hopf algebras and cotensor Hopf algebras. The aim of this subsection is to extend this result to tensor quasi-Hopf algebras and cotensor coquasi-Hopf algebras.

Before proceeding further, we introduce the notion of a quasi-coquasi-Hopf pairing, generalizing the notion of a Hopf pairing.
\begin{definition}\label{def:QCQpair}
Let $A=(A,m_A,\mu_A,\Delta_A,\varepsilon_A,\phi)$ and $H=(H,m_H,\mu_H,\Delta_H,\varepsilon_H,\Phi)$ be a quasi-bialgebra and a coquasi-bialgebra over $\k$, respectively. A quasi-coquasi-bialgebra pairing
is a bilinear map $\theta:A\times H\rightarrow \k$
such that
\begin{equation}
\theta( ab,h)=\theta( a,h_1)\theta( b,h_2), \label{eq:pair1}
\end{equation}
\begin{equation}
\theta( a,hg)=\theta( a_1,h)\theta( a_2,g),\label{eq:pair2}
\end{equation}
\begin{equation}
\theta( 1_A,h)=\varepsilon_H(h),\label{eq:pair3}
\end{equation}
\begin{equation}
\theta( a,1_H)=\varepsilon_A(a),\label{eq:pair4}
\end{equation}
\begin{equation}
\theta(\phi,h\otimes g\otimes l)=\Phi(h,g,l)\label{eq:pair5}
\end{equation}
for all $a,b\in A$ and $g,h,l\in H$. If $A$ and $H$ are respectively a quasi-Hopf algebra and a coquasi-Hopf algebra with chosen quasi-antipode data
$(S_A,\alpha_A,\beta_A)$ and $(S_H,\alpha_H,\beta_H)$,
the pairing is called a quasi-coquasi-Hopf pairing if, in addition,
\begin{equation}
\theta( S_A(a),h)=\theta( a,S_H(h)),\label{eq:pair6}
\end{equation}
\begin{equation}
\theta(\alpha_A,h)=\alpha_H(h),\label{eq:pair7}
\end{equation}
\begin{equation}
\theta(\beta_A,h)=\beta_H(h).\label{eq:pair8}
\end{equation}
The pairing is called right nondegenerate (respectively, left nondegenerate) if the condition $\theta(a,h)=0$ for all $a\in A$ forces $h=0$ (respectively, if $\theta(a,h)=0$ for all $h\in H$ forces $a=0$). It is called nondegenerate if it is both left and right nondegenerate. Assume further that $A = \bigoplus_{n\geq 0} A(n)$ and $H = \bigoplus_{n\geq 0} H(n)$ are graded. A quasi-coquasi-Hopf pairing $\theta : A \times H \rightarrow \k$ is said to be graded if $\theta( A(i), H(j))=0$ whenever $i\neq j$.
\end{definition}
\begin{example}\rm
\begin{itemize}
\item[(1)]Let $A$ be a finite-dimensional quasi-Hopf algebra, and let $H$ be a finite-dimensional coquasi-Hopf algebra. Let $$\langle -,-\rangle:H^*\times H\to \k$$ be the bilinear map defined by $\langle f,x\rangle=f(x)$. Then this bilinear map is a nondegenerate quasi-coquasi-Hopf pairing. Moreover, following the proof of \cite[Proposition V.7.2]{Kas95}, one can show that a nondegenerate quasi-coquasi-Hopf pairing
$$\langle -,-\rangle:A\times H\to \k$$
exists if and only if $H$ is isomorphic to $A^*$ as a coquasi-Hopf algebra.
\item[(2)]Let $A$ be a (infinite-dimensional) quasi-bialgebra. According to \cite[Subsection 5.2]{AES16}, the bilinear form
$$
\langle -,-\rangle:A\times A^\circ\to \k,
\qquad \langle a,f\rangle=f(a),
$$
is a right nondegenerate quasi-coquasi-bialgebra pairing. Here $A^\circ$ denotes the finite dual of $A$. However, from the coquasi-Hopf algebra side, the result is not completely parallel. Ardizzoni, El Kaoutit and Saracco \cite[Proposition 5.5]{AES16} showed that for a split coquasi-bialgebra $H$ (see \cite[Subsection 5.2]{AES16}), one can choose a good space $H^\bullet$ of $H^*$ (see \cite[Subsection 2.1]{AES16}) such that $H^\bullet$ admits a quasi-bialgebra structure and the pairing
$$
\langle-,-\rangle:H^\bullet\times H\rightarrow \k
$$
is a left nondegenerate quasi-coquasi-bialgebra pairing. They considered split coquasi-bialgebras in order to construct the associator on $H^\bullet$.
\end{itemize}
\end{example}

\begin{remark}\rm\label{rm:Ffpair}
Let $A$ be a quasi-Hopf algebra, $H$ a coquasi-Hopf algebra, and $\theta : A\times H \to \k$ a quasi-coquasi-Hopf pairing.
Let $F\in A\otimes A$ be the element defined in (\ref{eq:F}), and let $f:H\times H\to \k$ be the bilinear map defined as in (\ref{eq:f}). A direct computation shows that for any $h,g\in H$, we have
$$
\theta(F,h\otimes g)=f(h,g).
$$
\end{remark}
\begin{remark}\rm
Let $A$ be a quasi-Hopf algebra, $H$ a coquasi-Hopf algebra, and $\theta : A\times H \to \k$ a nondegenerate quasi-coquasi-Hopf pairing. If $R\in A\otimes A$ is a quasitriangular structure on $A$,
then we can show that
$$
r(x,y)=\theta( R,x\otimes y)
$$
is a coquasitriangular structure on $H$. If, in addition, $(A,R,v)$ is a ribbon quasi-Hopf algebra and we put
$$
\nu(x)=\theta( v,x),
$$
then $(H,r,\nu)$ is a coribbon coquasi-Hopf algebra. Conversely, let $r$ be a coquasitriangular structure on $H$.
If there exists an invertible element $R\in A\otimes A$
such that
$$
r(x,y)=\theta( R,x\otimes y),
$$
a direct computation shows that $R$ is a quasitriangular structure on $A$. Let $\nu$ be a coribbon functional on $(H,r)$.
If there exists an invertible element $v\in A$ such that
\begin{eqnarray*}
\nu(x)=\theta( v,x),
\end{eqnarray*}
then $v$ is a ribbon element of $(A,R)$.
\end{remark}

In what follows, we seek to establish graded quasi-coquasi-Hopf pairings between tensor quasi-Hopf algebras and cotensor coquasi-Hopf algebras.
Inspired by \cite[Subsection 3.2]{Che07}, we first introduce the following definition.
\begin{definition}
Let $A$ be a quasi-Hopf algebra and $H$ a coquasi-Hopf algebra. Let $M$ be a two-sided two-cosided quasi-Hopf $A$-bimodule and $N$ an $H$-Majid bimodule. A pairing between the pairs $(A,M)$ and $(H,N)$, denoted by
$$
(\theta_0,\theta_1):(A,M)\times(H,N)\rightarrow \k,
$$
is given by a quasi-coquasi-Hopf pairing $\theta_0:A\times H\rightarrow \k$ and a bilinear map $\theta_1:M\times N\rightarrow \k$ such that
\begin{equation}
\theta_1(a\cdot m\cdot b,n)=\theta_0(a,n_{-1})\theta_1(m,n_0)\theta_0(b,n_1), \label{eq:pairofmajid1}
\end{equation}
\begin{equation}
\theta_1(m,(h\cdot n)\cdot g)=\theta_0(m_{0,-1},h)\theta_1(m_{0,0},n)\theta_0(m_1,g), \label{eq:pairofmajid2}
\end{equation}
where $a,b\in A$, $g,h\in H$, $m\in M$, $n\in N$, and
$
(\rho^L_M \otimes\id_A)\rho^R_M(m)=m_{0,-1}\otimes m_{0,0}\otimes m_1.
$
\end{definition}
We note that (\ref{eq:pairofmajid2}) is quite different from its Hopf algebra counterpart, because $M$ is not a comodule in general, nor is $N$ a module. Moreover,
(\ref{eq:pairofmajid2}) is equivalent to
$$
\theta_1(m,h\cdot (n\cdot g))=\theta_0(m_{-1},h)\theta_1(m_{0,0},n)\theta_0(m_{0,1},g),
$$
where $(\id_A \otimes \rho^R_M)\rho^L_M(m)=m_{-1}\otimes m_{0,0}\otimes m_{0,1}.$

We are now ready to establish the main result of this subsection.
\begin{theorem}\label{thm:pair}
Let $A$ be a quasi-Hopf algebra and $H$ a coquasi-Hopf algebra. Let $M$ be a two-sided two-cosided quasi-Hopf $A$-bimodule and $N$ an $H$-Majid bimodule. Then a graded quasi-coquasi-Hopf pairing
$$
\theta:T_A(M)\times \operatorname{CoT}_H(N)\rightarrow \k
$$
exists if and only if there exists a pairing $(\theta_0,\theta_1)$ between $(A,M)$ and $(H,N)$. Moreover, $\theta$ is right nondegenerate if and only if $\theta_0$ and $\theta_1$ are right nondegenerate.
\end{theorem}
\begin{proof}
The ``only if'' part: Suppose that $\theta:T_A(M)\times \operatorname{CoT}_H(N)\rightarrow \k$ is a quasi-coquasi-Hopf pairing. A direct verification shows that $\theta_0:=\theta\mid_{A \times H}$ defines a quasi-coquasi-Hopf pairing between $A$ and $H$. Since $\theta$ is graded, for any $a, b\in A, g,h\in H, m\in M$, and $n\in N$, we have
\begin{eqnarray*}
\theta (a\cdot m\cdot b, n)&=& \theta(a\cdot m, n_{-1})\theta(b, n_0)+\theta(a\cdot m, n_0)\theta(b, n_1) \\
&=&\theta(a\cdot m, n_0)\theta(b, n_1)\\
&=& \theta(a, n_{-1}) \theta(m, n_0)\theta(b, n_1)+\theta(a, n_{0}) \theta(m, n_1)\theta(b, n_2)\\
&=&\theta(a, n_{-1}) \theta(m, n_0)\theta(b, n_1),
\end{eqnarray*}
and
\begin{eqnarray*}
\theta(m,(h\cdot n)\cdot g)&=& \theta(m_{-1}, h\cdot n)\theta(m_{0}, g)+\theta(m_{0}, h\cdot n)\theta(m_{1}, g)\\
&=&\theta(m_{0}, h\cdot n)\theta(m_{1}, g)\\
&=&\theta(m_{0,-1}, h) \theta(m_{0,0}, n)\theta(m_{1}, g)+\theta(m_{0,0}, h) \theta(m_{0,1}, n)\theta(m_{1}, g)\\
&=& \theta(m_{0,-1}, h) \theta(m_{0,0}, n)\theta(m_{1}, g).
\end{eqnarray*}
This means that if we define $\theta_1=\theta\mid_{M\times N}$, then $(\theta_0,\theta_1)$ is a pairing between $(A,M)$ and $(H,N)$. \\
The ``if " part: Suppose that $(\theta_0,\theta_1)$ is a pairing between $(A,M)$ and $(H,N)$. According to \cite[Lemma 3.3]{Che07}, there exists
a unique graded bilinear map $\theta:T_A(M)\times \operatorname{CoT}_H(N)\rightarrow \k$ extending $\theta_0$ and
$\theta_1$, which satisfies (\ref{eq:pair1}) and (\ref{eq:pair3}).
Similar to the proof of \cite[Proposition 3.4]{Che07}, we can show that $\theta$ also satisfies \eqref{eq:pair2} and \eqref{eq:pair4}. Indeed, the fact that $\theta$ is graded implies \eqref{eq:pair4}, while the uniqueness in \cite[Lemma 3.3]{Che07} shows that the following two bilinear maps coincide:
$$
\vartheta_1:T_A(M)\times \bigl(\operatorname{Cot}_H(N)\otimes \operatorname{Cot}_H(N)\bigr)\to \k,\quad (x,y\otimes z)\mapsto \theta(x,yz),
$$
and
$$
\vartheta_2:T_A(M)\times \bigl(\operatorname{Cot}_H(N)\otimes \operatorname{Cot}_H(N)\bigr)\to \k,\quad (x,y\otimes z)\mapsto \theta(x_1,y)\theta(x_2,z).
$$ Equations (\ref{eq:pair5}), (\ref{eq:pair7}), and (\ref{eq:pair8}) hold since, by definition, the elements $\tilde{\phi}$, $\tilde{\alpha}$, $\tilde{\beta}$ of $T_A(M)$ all lie in degree $0$, and $\Phi^\prime$, $\alpha^\prime$, $\beta^\prime$ of $\operatorname{CoT}_H(N)$ are concentrated in degree $0$, while $\theta_0$ is a quasi-coquasi-Hopf pairing. It remains only to prove (\ref{eq:pair6}). Note that by \cite[Proposition 2.16]{Yu26}, it suffices to consider elements of $N$ of the form $h\cdot n$, where $h\in H$ and $n\in {}^{co H}N$. Then for any $m\in M,$ we have
\begin{eqnarray*}
    \theta(\tilde{S}(m), n)&=&\theta(- S(m_{0,-1} \overline{X}^1) \alpha \cdot m_{0,0} \cdot \overline{X}^2 \beta S(\overline{X}^3) S(m_{1}),n)\\
    &=&\theta_0(- S_A(m_{0,-1} \overline{X}^1) \alpha_A, 1)\theta_1(m_{0,0}, n_0)\theta_0(\overline{X}^2 \beta_A S(\overline{X}^3) S_A(m_{1}), n_1)\\
    &=&-\varepsilon_A(m_{0,-1}) \varepsilon_A(\overline{X}^1)\theta_0(\alpha_A,1_H)\theta_1(m_{0,0}, n_0)\theta_0(\overline{X}^2 \beta_A S(\overline{X}^3) S_A(m_{1}), n_1)\\
    &=&-\theta_0(\alpha_A,1_H)\theta_1(m_0, n_0)\theta_0( \beta_A  S_A(m_{1}), n_1)\\
    &=&-\theta_0(\alpha_A,1_H)\theta_1(m_0, n_0)\theta_0( \beta_A  , n_1)\theta_0(S_A(m_{1}), n_2)\\
    &=&\theta (m, -\alpha_H(1_H)\beta_H(n_1)  n_0 S_H(n_2))\\
    &=& \theta(m, S^\prime (n)).
\end{eqnarray*}
Furthermore, let $F=F^1\otimes F^2$ be the element defined in (\ref{eq:F}), and we write $F^{-1}=\overline{F}^1\otimes \overline{F}^2.$
According to Remark \ref{rm:Ffpair}, we have
\begin{eqnarray*}
    \theta(\tilde{S}(m), h\cdot n)&=& \theta (\overline{F}^1 \tilde{S}(m_1) F^1, h) \theta (\overline{F}^2\tilde{S}(m_0) F^2, n)\\
    &=& \theta(\overline{F}^1, h_1)\theta(\tilde{S}(m_1) ,h_2) \theta (F^1, h_3) \theta (\overline{F}^2, 1) \theta (\tilde{S}(m_0), n_0)\theta(F^2,n_1)\\
    &=& \theta(F^{-1}, h_1\otimes 1)  \theta (m, S^\prime (n_0) S^\prime (h_2)) \theta(F, h_3\otimes n_1)\\
    &=&\theta (m, f^{-1}(h_1, 1)S^\prime (n_0) S^\prime (h_2) f(h_3, n_1))\\
    &=&\theta (m, S^\prime(h\cdot n)),
\end{eqnarray*}
where $f$ is the bilinear map defined as in (\ref{eq:f}).
Since $\theta$ is graded, it follows that $\theta\mid _{(A \oplus M)\otimes \operatorname{CoT}_H(N)}$ satisfy (\ref{eq:pair6}). We claim that whenever $\theta(-,c)$ satisfies (\ref{eq:pair6}) at $x$ and $y$ for all $c\in \operatorname{CoT}_H(N)$, it satisfies (\ref{eq:pair6}) at $xy$ as well. Indeed, we have
\begin{eqnarray*}
    \theta(\tilde{S} (xy),c)&=&\theta(\tilde{S} (y),c_1 )\theta(\tilde{S} (x),c_2 )\\
    &=&\theta(y,S^\prime (c_1) )\theta(x,S^\prime (c_2) )\\
    &=& \theta(xy, S^\prime (c)).
\end{eqnarray*}
Since $T_A(M)$ is generated as an algebra by $A$ and $M$, we can prove that $\theta$ fulfills (\ref{eq:pair6}). Thus $\theta$ is a quasi-coquasi-Hopf pairing. A similar argument as in the proof of \cite[Proposition 3.4]{Che07}, $\theta$ is right nondegenerate if and only if $\theta_0$ and $\theta_1$ are right nondegenerate.
\end{proof}
We remark that the graded quasi-coquasi-Hopf pairing $\theta$ defined in the proof of the above theorem, extending $\theta_0$ and $\theta_1$, is unique. Moreover, the restriction of the above pairing to $S_H(N)$ yields a corresponding result.
\begin{remark}\rm
With the same notions as in Theorem \ref{thm:pair},
let $Q$ be a graded quasi-Hopf algebra, and suppose that $q:T_A(M)\to Q$ is a graded quasi-Hopf algebra epimorphism with
$\operatorname{ker}(q)\subseteq \bigoplus_{n\geq 2} M^{\otimes_A n}.$ We claim that the bilinear form
$$
\varphi ( q(x), c):=\theta(x, c)
$$
forms a quasi-coquasi-Hopf pairing between $Q$ and the quasi-quantum symmetric algebra $S_H(N)$. Indeed, since $\theta$ is graded, it follows that $$\theta (\operatorname{ker}(q), H\oplus N)=0.$$
Define the set $$E=\{c\in \operatorname{CoT}_H(N)\mid \theta(x,c)=0 \text{ for all } x\in \operatorname{ker}(q)\}.$$
By (\ref{eq:pair2}) and the fact that $\tilde{\Delta}(\ker (q))\subseteq \ker (q)\underline{\otimes} T_A(M)+T_A(M)\underline{\otimes} \ker (q)$, $E$ is closed under multiplication; hence we have
$S_H(N)\subseteq E$. This means that $\varphi$ is well-defined. Using the fact that $q$ is a graded quasi-Hopf algebra epimorphism, one can prove directly that $\varphi$ is a quasi-coquasi-Hopf pairing. This result generalizes \cite[Theorem 3.1]{Che07}. It is worth noting that, if we take $V$ to be the alternative coinvariant $M^{\overline{co A}}$ (see \cite[Remark 3.4]{BT06} for the definition; see also \cite[Proposition 2.1]{Bul20}), which carries a left-left Yetter-Drinfeld module structure over $A$, then the quasi-Hopf algebra $\mathcal{B}(V)\#A$ arises as a special case of $Q$. For the definition of a left-left Yetter-Drinfeld module over a quasi-Hopf algebra, see \cite[Definition 5.1]{Sch02}; for bosonization of quasi-Hopf algebras, see \cite{BN02}.
\end{remark}

\section*{Acknowledgement}
AI was used solely for language polishing and spell-checking. The authors take full responsibility for the content of the paper. The first author was supported by the China Postdoctoral Science Foundation [Grant number 2026M793410] and the Postdoctoral Fellowship Program of China Postdoctoral Science Foundation [Grant number GZC20261673].


\begin{thebibliography}{99}
\bibitem{AC92}Altsch\"uler, D., and Coste, A. ``Quasi-quantum groups, knots, three-manifolds, and topological field theory." \textit{Commun. Math. Phys.} 150, no. 1 (1992): 83-107.
\bibitem{AC13}Andruskiewitsch, N., and Cuadra, J. ``On the structure of (co-Frobenius) Hopf algebras." \textit{J. Noncommut. Geom.} 7, no. 1 (2013): 83-104.
\bibitem{ACE15}Andruskiewitsch, N., Cuadra, J., and Etingof, P. ``On two finiteness conditions for Hopf algebras with nonzero integral." \textit{Ann. Sc. Norm. Super. Pisa Cl. Sci.} (5) 14, no. 2 (2015): 401-440.
\bibitem{AS10}Andruskiewitsch, N., and Schneider, H.-J. ``On the classification of finite-dimensional pointed Hopf algebras." \textit{Ann. of Math.} (2) 171, no. 1 (2010): 375-417.
\bibitem{Ang26}Angiono, I. ``Pointed Hopf algebras revisited, with a view from tensor categories." In \textit{International Congress of Mathematicians 2026}, Vol. 3, 153-171. 2026.
\bibitem{AES16}Ardizzoni, A., El Kaoutit, L., and Saracco, P. ``Functorial constructions for non-associative algebras with applications to quasi-bialgebras." \textit{J. Algebra} 449 (2016): 460-496.
\bibitem{AMS07}Ardizzoni, A., Menini, C., and \c{S}tefan, D. ``Hochschild cohomology and ``smoothness'' in monoidal categories." \textit{J. Pure Appl. Algebra} 208 (2007): 297-330.
\bibitem{AP12}Ardizzoni, A., and Pavarin, A. ``Preantipodes for dual quasi-bialgebras." \textit{Israel J. Math.} 192, no. 1 (2012): 281-295.
\bibitem{AP13}Ardizzoni, A., and Pavarin, A. ``Bosonization for dual quasi-bialgebras and preantipode." \textit{J. Algebra} 390 (2013): 126-159.
\bibitem{AM19}Aziz, R., and Majid, S. ``Co-double bosonisation and dual bases of $c_q[SL_2]$ and $c_q[SL_3]$." \textit{J. Algebra} 518 (2019): 75-118.
\bibitem{Bal09}Balan, A. ``Yetter-Drinfeld modules and Galois extensions over coquasi-Hopf algebras." \textit{Politehn. Univ. Bucharest Sci. Bull. Ser. A Appl. Math. Phys.} 71, no. 3 (2009): 43-60.
\bibitem{BB13}Bazlov, Y., and Berenstein, A. ``Cocycle twists and extensions of braided doubles." In \textit{Noncommutative Birational Geometry, Representations and Combinatorics}, 19- 70. Contemporary Mathematics 592. Providence, RI: American Mathematical Society, 2013.
\bibitem{Bul20}Bulacu, D. ``Quasi-quantum groups obtained from tensor braided Hopf algebras." \textit{J. Algebraic Combin.} 52, no. 3 (2020): 405-453.
\bibitem{BC03}Bulacu, D., and Caenepeel, S. ``Integrals for (dual) quasi-Hopf algebras. Applications." \textit{J. Algebra} 266 (2003): 552-583.
\bibitem{BCPV19}Bulacu, D., Caenepeel, S., Panaite, F., and Van Oystaeyen, F. \textit{Quasi-Hopf Algebras: A Categorical Approach}. Cambridge: Cambridge University Press, 2019.
\bibitem{BN02}Bulacu, D., and Nauwelaerts, E. ``Radford's biproduct for quasi-Hopf algebras and bosonization." \textit{J. Pure Appl. Algebra} 174, no. 1 (2002): 1-42
\bibitem{BT06}Bulacu, D., and Torrecillas, B. ``Two-sided two-cosided Hopf modules and Yetter-Drinfeld modules for quasi-Hopf algebras." \textit{Appl. Categ. Struct.} 14, no.5-6 (2006): 503-530.
\bibitem{Che07}Chen, X.-W. ``Duality between quantum symmetric algebras." \textit{Lett. Math. Phys.} 79, no. 1 (2007): 39-50.
\bibitem{CR02}Cibils, C., and Rosso, M. ``Hopf quivers." \textit{J. Algebra} 254, no. 2 (2002): 241-251.
\bibitem{CGR20}Creutzig, T., Gainutdinov, A.M., and Runkel, I. ``A quasi-Hopf algebra for the triplet vertex operator algebra." \textit{Commun. Contemp. Math.} 22, no. 3 (2020): 1950024.
\bibitem{Cli19}Cline, Z. ``On actions of Drinfel'd doubles on finite dimensional algebras." \textit{J. Pure Appl. Algebra} 223, no. 8 (2019): 3635-3664.
\bibitem{DPR90}Dijkgraaf, R., Pasquier, V., and Roche, P. ``Quasi Hopf algebras, group cohomology and orbifold models." \textit{Nucl. Phys. B, Proc. Suppl.} 18B (1990): 60-72.
\bibitem{DNR01}D\u{a}sc\u{a}lescu, S., N\u{a}st\u{a}sescu, C., and Raianu, S. \textit{Hopf Algebras. An Introduction}. New York: Marcel Dekker, Inc., 2001.
\bibitem{Dri87}Drinfeld, V.G. ``Quantum groups." In \textit{Proceedings of the International Congress of Mathematicians}, Vol. 1, 2, 798-820. Providence, RI: Amer. Math. Soc., 1987.
\bibitem{Dri90}Drinfeld, V.G. ``Quasi-Hopf algebras." \textit{Algebra Anal.} 1, no. 6 (1989): 114-148; translation in \textit{Leningr. Math. J.} 1, no. 6 (1990): 1419-1457.
\bibitem{DNR25}Dong, C., Ng, S.-H., and Ren, L. ``Orbifolds and minimal modular extensions." \textit{Adv. Math.} 462 (2025): 110103.
\bibitem{EG03}Etingof, P., and Gelaki, S. ``The classification of finite-dimensional triangular Hopf algebras over an algebraically closed field of characteristic $0$." \textit{Mosc. Math. J.} 3, no. 1 (2003): 37-43, 258.
\bibitem{EG05}Etingof, P., and Gelaki, S. ``On radically graded finite-dimensional quasi-Hopf algebras." \textit{Mosc. Math. J.} 5, no. 2 (2005): 371-378.
\bibitem{EGNO15}Etingof, P., Gelaki, S., Nikshych, D., and Ostrik, V. \textit{Tensor Categories}. Mathematical Surveys and Monographs 205. Providence, RI: American Mathematical Society, 2015.
\bibitem{EG21}Etingof, P., and Gelaki, S. ``Finite symmetric integral tensor categories with the Chevalley property, with an appendix by Kevin Coulembier and Pavel Etingof." \textit{Int. Math. Res. Not.} no. 12 (2021): 9083-9121.
\bibitem{EO04}Etingof, P., and Ostrik, V. ``Finite tensor categories." \textit{Mosc. Math. J.} 4, no. 3 (2004): 627-654.
\bibitem{EP26}Etingof, P., and Penneys, D. ``Rigidity of non-negligible objects of moderate growth in braided categories." \textit{Forum Math. Pi} 14 (2026): Paper No. e7, 18 pp.
\bibitem{FJ15}Fang, X., and Jian, R.-Q. ``Cofree Hopf algebras on Hopf bimodule algebras." \textit{J. Pure Appl. Algebra} 219, no. 9 (2015): 3913-3930.
\bibitem{FR17}Fang, X., and Rosso, M. ``Multi-brace cotensor Hopf algebras and quantum groups." \textit{Math. Z.} 286 (2017): 657-678.
\bibitem{FGR22}Farsad, V., Gainutdinov, A.M., and Runkel, I. ``The symplectic fermion ribbon quasi-Hopf algebra and the $SL(2, \mathbb{Z})$-action on its centre." \textit{Adv. Math.} 400 (2022): Paper No. 108247, 87 pp.
\bibitem{GS98}Green, E.L., and Solberg, {\O}. ``Basic Hopf algebras and quantum groups." \textit{Math. Z.} 229 (1998): 45-76.
\bibitem{Hua09}Huang, H.-L. ``Quiver approaches to quasi-Hopf algebras." \textit{J. Math. Phys.} 50, no. 4 (2009): Paper No. 043501, 9 pp.
\bibitem{HLYY20}Huang, H.-L., Liu, G., Yang, Y., and Ye, Y. ``Finite quasi-quantum groups of diagonal type." \textit{J. Reine Angew. Math.} 759 (2020): 201-243.
\bibitem{HLYY26}Huang, H.-L., Liu, G., Yang, Y., and Ye, Y. ``On the classification of finite quasi-quantum groups over abelian groups." \textit{Adv. Math.} 486 (2026): Paper No. 110740, 33 pp.
\bibitem{HLY11}Huang, H.-L., Liu, G., and Ye, Y. ``Quivers, quasi-quantum groups and finite tensor categories." \textit{Comm. Math. Phys.} 303, no. 3 (2011): 595-612.
\bibitem{Kas95}Kassel, C. \textit{Quantum Groups}. Graduate Texts in Mathematics 155. New York: Springer-Verlag, 1995.
\bibitem{Liu14}Liu, G. ``The quasi-Hopf analogue of $u_q(\mathfrak{sl}_2)$." \textit{Math. Res. Lett.} 21, no. 3 (2014): 585-603.
\bibitem{LVZ17}Liu, G., Van Oystaeyen, F., and Zhang, Y. ``Quasi-Frobenius-Lusztig kernels for simple Lie algebras." \textit{Trans. Amer. Math. Soc.} 369, no. 3 (2017): 2049-2086.
\bibitem{Lus90}Lusztig, G. ``Finite dimensional Hopf algebras arising from quantized universal enveloping algebras." \textit{J. Amer. Math. Soc.} 3 (1990): 257-296.
\bibitem{Maj95}Majid, S. \textit{Foundations of Quantum Group Theory}. Cambridge: Cambridge University Press, 1995.
\bibitem{Mon93}Montgomery, S. \textit{Hopf Algebras and Their Actions on Rings}. CBMS Regional Conference Series in Mathematics 82. Providence, RI: American Mathematical Society, 1993.
\bibitem{Nic78}Nichols, W. ``Bialgebras of type one." \textit{Comm. Algebra} 6, no. 15 (1978): 1521-1552.
\bibitem{Rad75}Radford, D.E. ``On a coradical of a finite-dimensional Hopf algebra." \textit{Proc. Amer. Math. Soc.} 53, no. 1 (1975): 9-15.
\bibitem{Rad77}Radford, D.E. ``Operators on Hopf algebras." \textit{Amer. J. Math.} 99, no. 1 (1977): 139-158.
\bibitem{Ros98}Rosso, M. ``Quantum groups and quantum shuffles." \textit{Invent. Math.} 133, no. 2 (1998): 399-416.
\bibitem{Run14}Runkel, I. ``A braided monoidal category for free super-bosons." \textit{J. Math. Phys.} 55 (2014): Paper No. 041702, 59 pp.
\bibitem{Sch02}Schauenburg, P. ``Hopf modules and the double of a quasi-Hopf algebra." \textit{Trans. Amer. Math. Soc.} 354, no. 8 (2002): 3349-3378.
\bibitem{Tak77}Takeuchi, M. ``Morita theorems for categories of comodules." \textit{J. Fac. Sci. Univ. Tokyo Sect. IA Math.} 24, no. 3 (1977): 629-644.
\bibitem{VZ04}Van Oystaeyen, F., and Zhang, P. ``Quiver Hopf algebras." \textit{J. Algebra} 280 (2004): 577-589.
\bibitem{SY26}Shimizu, K., and Yadav, H. ``Ribbon categories from ind-exact algebras: simple current case." arXiv:2603.28215.
\bibitem{Wa03}Wakui, M. ``The coribbon structures of some finite dimensional braided Hopf algebras generated by 2$\times$2-matrix coalgebras." In \textit{Noncommutative Geometry and Quantum Groups (Warsaw, 2001)}, 333- 344. Warsaw: Polish Acad. Sci. Inst. Math., 2003.
\bibitem{Yu26}Yu, J. ``A quiver approach to quasi-quantum groups with the Chevalley property." arXiv:2606.10946.
\end{thebibliography}
\end{document}